\documentclass[10pt, twoside, notitlepage]{amsart}

\usepackage[utf8]{inputenc}
\usepackage{amssymb}
\usepackage{amsmath}
\usepackage{enumerate}
\usepackage{amsfonts}
\usepackage{color}
\usepackage{latexsym}
\usepackage{mathtools}
\usepackage{amstext,amssymb,amsopn,amsthm,mathrsfs,wasysym,dsfont,comment,bigints,tensor,upgreek}
\usepackage{longtable,tabularx,booktabs}

\usepackage{geometry}
\usepackage{graphicx}
\usepackage{esint}

\usepackage{float}
\usepackage{textcomp}

\usepackage[dvipsnames]{xcolor}
\usepackage{tikz}
\usetikzlibrary{patterns,arrows}
\usepackage{cite}

\usepackage{pgfplots}
\pgfplotsset{compat=1.15}
\usepackage{mathrsfs}
\usepackage[
colorlinks=true,
linkcolor=magenta,
citecolor=magenta,
urlcolor=magenta,
]{hyperref}

\numberwithin{equation}{section}

\definecolor{orange}{rgb}{1.0, 0.5, 0.0}

\theoremstyle{plain}
\newcounter{counter}

\newtheorem{thm}{Theorem}
\newtheorem{theorem}[counter]{Theorem}
\newtheorem{prop}{Proposition}[section]
\newtheorem{lem}[prop]{Lemma}
\newtheorem{cor}[prop]{Corollary}

\newtheorem{rmk}[prop]{Remark}

\newtheorem{example}[prop]{Example}

\newcommand {\R} {\mathbb{R}} \newcommand {\Z} {\mathbb{Z}}
 \newcommand {\N} {\mathbb{N}}
 
\newcommand {\Zd} {\mathbb{Z}^d}

\DeclareMathOperator{\supp}{supp}

\DeclareMathOperator {\Ree} {Re}

\DeclareMathOperator{\dis} {\operatorname{d}}

\allowdisplaybreaks

\definecolor{ffqqqq}{rgb}{1,0,0}
\definecolor{ccqqqq}{rgb}{0.8,0,0}
\definecolor{ttqqqq}{rgb}{0.2,0,0}
\definecolor{qqqqff}{rgb}{0,0,1}

\title{Landis-type results for discrete equations}

\date{\today}

\author[A. Fern\'andez-Bertolin]{Aingeru Fern\'andez-Bertolin} \address[Aingeru Fern\'andez-Bertolin]{Facultad de Ciencia y Tecnolog\'ia, Universidad del Pa\'is Vasco / Euskal Herriko Unibertsitatea (UPV/EHU), Departamento de Matem\'aticas, UPV/EHU, Apartado 644, 48080 Bilbao, Spain}
\email{\href{mailto:aingeru.fernandez@ehu.eus}{\textnormal{aingeru.fernandez@ehu.eus}}}

\author[L.\ Roncal]{Luz Roncal}
\address[Luz Roncal]{
BCAM - Basque Center for Applied Mathematics,
48009 Bilbao, Spain
and
Ikerbasque, Basque Foundation for Science,
48011 Bilbao, Spain
and
Universidad del Pa\'is Vasco / Euskal Herriko Unibertsitatea,
Apartado 644, 48080 Bilbao, Spain}
\email{\href{mailto:lroncal@bcamath.org}{\textnormal{lroncal@bcamath.org}}}

\author[D. Stan]{Diana Stan}
\address[D. Stan]{University of Cantabria, Department of Mathematics, Statistics and Computation,
Avd. Los Castros 44, 39005 Santander, Spain}
\email{\href{mailto:diana.stan@unican.es}{\textnormal{diana.stan@unican.es}}}



\begin{document}


\thanks{A. Fern\'andez-Bertolin is partially supported by
project PID2021-122156NB-I00 (AEI/FEDER, UE) and acronym HAMIP, and
the Basque Government through the project IT1615-22.}

\thanks{L. Roncal is partially supported by the projects CEX2021-001142-S, RYC2018-025477-I and PID2020-113156GB-I00/AEI/10.13039/501100011033 with acronym ``HAPDE'' funded by Agencia Estatal de Investigaci\'on, grant BERC 2022-2025 of the Basque Government and IKERBASQUE}

\thanks{D. Stan is supported by the project  PID2020-114593GA-I00 financed by MCIN/AEI/10.13039/501100011033}

	\begin{abstract} We prove Landis-type results for both the semidiscrete heat and the stationary discrete Schr\"odinger equations. For the semidiscrete heat equation we show that, under the assumption of two-time spatial decay conditions on the solution $u$, then necessarily $u\equiv 0$. For the stationary discrete Schr\"odinger equation we deduce that, under a vanishing condition at infinity on the solution $u$, then $u\equiv 0$.  In order to obtain such results, we demonstrate suitable quantitative upper and lower estimates for the $L^2$-norm of the solution within a spatial lattice $(h\Z)^d$. These estimates manifest an interpolation phenomenon between continuum and discrete scales, showing that close-to-continuum and purely discrete regimes are different in nature.
		
		\end{abstract}

\maketitle


\section{Introduction and main results}

Let us consider the stationary Schr\"odinger equation
\begin{equation}
\label{eq:SE}
(-\Delta+V)u=0 \quad \text{ in } \R^d,
\end{equation}
with $|V(x)|\le 1$. The question about the maximal rate of decay of a solution to \eqref{eq:SE} was raised by Landis in 1960 (see \cite{KL88}). The so-called \textit{Landis' conjecture} states that if a solution  of the equation \eqref{eq:SE}  satisfies $|u(x)|\le \exp{(-C|x|^{1+})}$, then $u\equiv 0$.
Meshkov \cite{Meshkov} disproved the conjecture by constructing a \textit{complex-valued} potential $V$ and a nontrivial solution $u$ to \eqref{eq:SE} with $|u(x)|\le \exp{(-C|x|^{\frac43})}$. Moreover, he showed that if  $|u(x)|\le \exp{(-C|x|^{\frac43+})}$, then $u\equiv 0$. In words, the exponent $4/3$ is optimal in the complex case. Later on, Bourgain and Kenig \cite{BK05} showed a quantitative form of Meskhov's result which was in connection with the resolution of Anderson localization for the Bernoulli model in higher dimensions. We also mention the complex counterexample by Cruz-Sampedro \cite{CS99} that improves Meskhov's result.

A refinement of Landis' conjecture was presented by Kenig \cite{K05}, who raised the question about the validity of the conjecture for \textit{real-valued} potentials and solutions. A partial positive answer was given by Kenig, Silvestre, and Wang \cite{KSW15}, who proved that a quantitative form of the conjecture is true in the plane ($d=2$) for real-valued $V\ge 0$. Landis' conjecture in the real case with $d=1$ was studied by Rossi \cite{R21}. Recently, Logunov, Malinnikova, Nadirashvili, and Nazarov \cite{LMNN20} showed that Landis' conjecture is true for $d=2$ and real potentials, and Davey \cite{Davey23} extended the result also for $d=2$ and real potentials, by allowing the potentials to grow. Up to our knowledge, the Landis conjecture in its general form in higher dimensions $d\ge 3$ remains still an open question.

Let us consider now the time-dependent Schr\"odinger equation
\begin{equation}
\label{fS}
 \partial_tu(t,x)=i(\Delta u+Vu).
\end{equation}
In a series of works, Escauriaza, Kenig, Ponce, and Vega \cite{EKPV-CPDE, EKPV-JEMS, EKPV-DUKE, EKPV-CMP, EKPV-BAMS} proved that if $V$ satisfies one of the following conditions
\begin{itemize}
\item[(i)] $\lim_{R\to\infty}\int_0^T\sup_{|x|>R}|V(t,x)|\,dt=0$
\item[(ii)] $V(t,x)=V_1(x)+V_2(t,x)$, where $V_1$ is real-valued and for some positive $\alpha$ and $\beta$,
$$ \sup_{[0,T]}\|e^{\alpha\beta T^2|x|^2/(\sqrt{\alpha}t+\sqrt{\beta}(T-t))^2}V_2(t)\|_{L^{\infty}(\R^d)}<+\infty,
$$
\end{itemize}
and we assume that $u$ is a solution to \eqref{fS} which fulfills the decay conditions $ |u(0,x)|\le Ce^{-\alpha|x|^2}$,
$ |u(T,x)|\le Ce^{-\beta|x|^2}
$
with $\alpha\beta>1/(16T^2)$, then $u\equiv0$. Similarly, for the heat equation with potential, the following result can be deduced: let
$V(t,x)\in L^{\infty}(\R\times \R^d)$ and $u$ be a solution to $\partial_tu=\Delta u+Vu$; if
$ |u(T,x)|\le Ce^{-\delta|x|^2}
$
and $\delta>1/(4T)$, then $u\equiv0$. The conditions on $\alpha\beta$ and $\delta$ are sharp. These results can be understood as dispersive and parabolic analogues of Landis' conjecture.

In this paper, we are interested in: $(1)$ studying analogous Landis-type results to the ones described above when we consider equations involving discrete Laplacians in a mesh of size $h>0$, and $(2)$ analysing the different behaviour of the solutions as long as the size of the mesh is shrinking to zero.
In general, we refer to \textit{Landis-type results} when we are interested in the maximum vanishing rate of solutions to equations with potentials, namely:
\begin{itemize}
\item In the case of elliptic equations, we are concerned about the maximal rate at which nontrivial solutions vanish at infinity.
\item Concerning dispersive or parabolic equations, we are concerned about the maximal spatial decay rate  of nontrivial solutions when time varies whithin a bounded interval.
\end{itemize}

Let us consider the discrete lattice $(h\Z)^d$ with $h\in \R_+$, $d\ge1$. Given a function $f:(h\Z)^d\to \R$, we define the discrete Laplacian as
\begin{equation}
\label{eq:delta}
\Delta_{\dis} f_j:=\frac{1}{h^2}\sum_{k=1}^d \big(f_{j+e_k}-2f_j+f_{j-e_k}\big), \quad j\in \Z^d,
\end{equation}
where we denote $f_j:=f(hj)$ and $e_k$ is the unit vector in the $k$-th direction. Landis-type uniqueness theorems for the time-dependent Schr\"odinger equation with a discrete Laplacian were studied in \cite{FB19, FB18, BFV17, JLMP18}.  We consider the equation, in a mesh of size $h=1$,
\begin{equation}\label{ph1}
\partial_t u_j=i(\Delta_{\dis} u_j + V_j u_j) \quad   \text{in } \mathbb{Z}^d\times \R_+,
\end{equation}
where $V:\Z^d\times \R_+\to \R$ is a bounded potential and we use the notation $u_j=u_j(t):=u(j,t)$. It was proved in \cite{JLMP18} (for $d=1$) and in \cite{BFV17} (for arbitrary dimensions) that, if $u$ is a solution to \eqref{ph1} and there exists a constant $\gamma$ such that
$$
|u(j,0)|+|u(j,1)|\le C \exp(-\gamma |j|\log |j|), \quad j\in \Z^d\setminus\{0\},
$$
then $u=0$.

Estimates on the decay of stationary solutions of discrete Schr\"odinger operators
$$
\Delta_{\dis} u_j + V_j u_j=0 \quad   \text{in } \mathbb{Z}^d,
$$
where $V:\Z^d\to \R$ is a bounded potential and $u_j:=u(j)$, and sharp uniqueness results for this equation, were obtained in \cite{LM18} by Lyubarskii and Malinnikova: if $u(x)$ satisfies the following decay estimate
$$
\liminf\limits_{N\rightarrow\infty}\frac{\ln(\max_{|n|_{\infty}\in \{N,N+1\}}|u(n)|)}{N}<-\|V\|_{\infty}-4d+1
$$
where $|n|_{\infty}=\max\{n_1,\ldots,n_d\}$, for $n\in \Z^d$, then $u\equiv 0$.

Comparing these results to the continuum versions, we see that the decay assumption is much weaker in the discrete results, which, intuitively, points out to the fact that the discretization process deteriorates the fast decay of the solutions at large scales. The goal of this paper is to show Landis-type results for the semidiscrete heat equation and the stationary discrete  Schr\"odinger equation with bounded potentials, in a mesh $(h\Z)^d$, for $h>0$ and $d\ge1$, and through our results we want to connect the different behaviors above mentioned, studying the discrete solutions at different scales.  More concretely, we will deal with the following problems:
\begin{enumerate}
\item the semidiscrete heat equation
\begin{equation}\label{p1}
\partial_t u_j=\Delta_{\dis} u_j + V_j u_j \quad   \text{in } (h\mathbb{Z})^d\times \R_+ \qquad u_j(0)=\psi_j \quad \text{ in }  (h\mathbb{Z})^d,
\end{equation}
where $V:(h\Z)^d\times \R_+\to \R$ is a bounded potential, $u:(h\Z)^d\times \R_+\to \R$, and $\psi:(h\Z)^d\to \R$.
\item  the stationary discrete  Schr\"odinger equation
\begin{equation}\label{p2intro}
\Delta_{\dis} u_j + V_j u_j=0 \quad   \text{in } (h\mathbb{Z})^d,
\end{equation}
where $V:(h\Z)^d\times \R_+\to \R$ is a bounded potential and $u:(h\Z)^d\to \R$.
\end{enumerate}
Notice that for the equation \eqref{p1} we use the notation $u_j=u_j(t):=u(hj,t)$, whereas for \eqref{p2intro} the notation is $u_j:=u(hj)$. The different meaning with the notation involving only the spatial lattice component will be clear from the context.

The main results of the paper are Theorems \ref{thm:qualitative} and \ref{thm:Schr}.
 An interesting, new contribution is that the results will be obtained through quantitative estimates within a spatial lattice $(h\Z)^d$ which manifest an interpolation phenomenon between continuum and discrete scales. In the case of the semidiscrete heat equation, under assumption of a two-time spatial decay condition on the solution, we conclude that the solution is trivial, see Theorem \ref{thm:qualitative}; for the stationary discrete Schr\"odinger equation, we prove that if a solution decays at a certain rate, then the solution is trivial, see Theorem \ref{thm:Schr}.

\subsection{Qualitative estimates}

Our first main result will show qualitative behaviour of solutions to \eqref{p1}, assuming a two-time decaying condition. Such a condition is given in terms of Macdonald's functions $K_{\nu}(x)$ (see Appendix \ref{sub:Bessel} for the definition). We denote\footnote{Along the paper, we will be considering an interval of time $t\in [0,1]$. All the results can be adapted to an interval $[0,T]$ for a fixed $T>0$; in that case there will be constants which will depend on $T$.}
$$
\|V\|_{\infty}:=\sup_{j\in \Z^d, \,t\in [0,1]}\{|V_j(t)|\}, \qquad  \|u\|_2:=\Big(h^d\int_0^1\sum_{j\in \Z^d}|u_j(t)|^2\,dt\Big)^{1/2}.
$$

\begin{theorem}[Landis-type result for the semidiscrete heat equation]
\label{thm:qualitative}
Let $h>0$ and $u\in C^1([0,1]:\ell^2((h\Z)^d))$ be a solution to
$$
\partial_t u_j=\Delta_{\dis} u_j + V_j u_j \quad   \text{in } (h\mathbb{Z})^d\times \R_+,
$$
 where $\|V\|_{\infty}$, $\|u\|_2$ are finite and independent of $h$.
 \begin{enumerate}
\item  Let $\gamma$ be a positive constant and assume that
  $$
 h^d \sum_{j\in \Z^d} \prod_{k=1}^d\Big(\frac{K_{j_{k}}^2(\frac{\gamma}{h^2})}{K_{0}^2(\frac{\gamma}{h^2})}\Big)\big(|u_j(0)|^2  +
|u_j(1)|^2\big) < \infty
  $$
  uniformly with respect to $h$.
Then there exist $C, h_0>0$ such that if $\gamma<C$ and $h\in (0,h_0)$, $u\equiv 0$.

\item
There exists $\mu_0=\mu_0(d)$ such that if, for $\mu>\mu_0$
$$
h^d \sum_{j\in \Z^d} \prod_{k=1}^d \Big( \frac{K_{j_k\mu}(\frac{2}{eh^2})}{K_{0}(\frac{2}{eh^2})}\Big)\big(|u_j(0)|^2  +
|u_j(1)|^2\big)<\infty,
$$
 then $u\equiv 0$.
\end{enumerate}
\end{theorem}

The uniqueness property in $(1)$ of Theorem \ref{thm:qualitative} corresponds with a situation in which the mesh is shrinking to the continuum Euclidean setting. Up to our knowledge, this situation had not been explored so far. On the other hand, property $(2)$ is linked to a purely discrete regime, where we should understand $h$ as a fixed parameter. The parameter $\gamma$ involved in $(1)$  is not playing a relevant role in the purely discrete regime in the sense that there cannot be a suitable choice of $\gamma$ which produces a contradiction in the reasoning, leading to the desired conclusion (this will become clear in Subsection \ref{subsec:upperb}, see also Remark \ref{rmk:Rhgrande}). This forces to consider a different condition in $(2)$ for the purely discrete regime. This result in $(2)$ is actually a slightly improved and scaled version of \cite[Theorem 1.2]{BFV17} (see also  \cite{JLMP18}).

We want to point out that if $\mu_0$ in $(2)$ of Theorem \ref{thm:qualitative}, a quantity that we are not quantifying, is seen to be small enough, then one could prove the first part of the theorem, for any $\gamma$, as a consequence of the second part. However, we have decided to state and prove the two parts of Theorem \ref{thm:qualitative} separately to stress out the results one can get by analyzing, on the one hand, the close-to-continuum setting, which resembles the decay condition given by Escauriaza, Kenig, Ponce and Vega, and, on the other hand, the purely discrete setting.

The second main result concerns the elliptic equation \eqref{p2intro}.

\begin{theorem}[Landis-type result for the stationary discrete Schr\"odinger equation]\label{thm:Schr}
Let $h>0$ and $u\in \ell^2((h\mathbb{Z})^d)$ be a solution to
$$
\Delta_{\dis} u_j + V_j u_j=0
$$
where $\|V\|_{\infty}$, $\|u\|_2$ are finite and independent of $h$.
\begin{enumerate}
\item There exists $\mu_0=\mu_0(d)$ such that if
$$
h^d\sum_{j\in \Z^d} e^{\mu_0 |jh|^{4/3}}  |u_j|^2 <\infty,
$$
uniformly with respect to $h$, then there exists $h_0>0$ such that if $h\in(0,h_0)$, $u\equiv0$.
\item There exists $\mu_0=\mu_0(d)$ such that if, for some $\beta>3$
$$
h^d\sum_{j\in \Z^d}  e^{\mu_0 |jh|^{1+\frac{1}{\beta}}\log |j|h} |u_j|^2 < \infty,
$$
uniformly with respect to $h$, then there exists $h_0>0$ such that $h\in (0,h_0)$ implies $u\equiv0$.

\item There exists $\mu_0=\mu_0(d)$ such that if,
$$
h^d\sum_{j\in \Z^d}  e^{\mu_0 |j|\log |j|h} |u_j|^2 < \infty,
$$
then $u\equiv0$.
\end{enumerate}
\end{theorem}

Let us discuss the different statements in Theorem \ref{thm:Schr}. The first part concerns the close-to-continuum setting and it will be derived from the study of nontrivial solutions close to the continuum, posing our problem in a mesh of size $h$ which shrinks to zero. We observe that we get the same exponent as in \cite{BK05, Meshkov}. The third part, for which $h$ is considered a fixed parameter, concerns the purely discrete setting, and the result can be compared with the uniqueness result in \cite{LM18}. The second part lies in between these two regimes as we will derive these result by understanding the behavior of nontrivial solutions at scales $h^{-\beta}$ for $\beta$ large. While we are still posing our problem in a lattice that shrinks to zero, we are looking at our solutions in a region that is far from the close-to-continuum setting. As we can see in the result, what we deduce is that an interpolating decay between the continuum decay (exponent $4/3$) and the purely discrete decay (exponent 1 up to a logarithm) provides unique continuation. As it happens in the parabolic setting, we can see these theorems as a cascade of results, since $(3)$ implies $(2)$ and $(2)$ implies $(1)$. Since our goal is to understand the effect of the discretization on the continuum result, we have decided to state and prove each part separately because it clearly shows how the largest decay possible for a nontrivial solution deteriorates as we escape the close-to-continuum region.

 The strategy of the proofs follows the scheme of the continuum results in  \cite{EKPV-BAMS}; similar approach was also exploited in \cite{BFV17}. In the case of Theorem \ref{thm:qualitative}, the first step  (inspired in the classical approach used by Agmon for elliptic equations \cite{A65}) is to show logarithmic convexity estimates for certain weighted norm of the solution to \eqref{p1}. With this logarithmic convexity at hand,  the decay conditions at times $t=0$ and $t=1$ imply upper bounds for the $\ell^2$-norm of the solution localized to a region of size $|hj|\simeq R$ with $R$ large. By distinguishing two regimes, we deduce upper bounds, under slightly different conditions, within the close-to-continuum and the purely discrete regime. The two scenarios arise naturally when we study the second ingredient, which are suitable lower bounds obtained through a Carleman inequality.

 Indeed, the second step is to prove a Carleman-type inequality, whose proof relies on the computation of a commutator between a symmetric and antisymmetric conjugate operator involving a parameter (Carleman parameter).
 Thanks to this inequality, one can deduce lower bounds for nontrivial solutions of \eqref{p1} with a general bounded potential. In order to do that, one considers suitable cut-off functions and defines a function to which the action of the commutator is computed. The only assumption to obtain the lower bound is that the solution is nontrivial. It is at this point where the quantitative nature of the problem plays a role and it forces to the appearance of two cases: the lower bounds involve hyperbolic functions with arguments depending on the mesh size. We linearize these functions by using the corresponding asymptotics, which lead to the dichotomy of conditions and to the two regimes: the close-to-continuum scenario and the purely discrete regime.

Finally, with an  appropriate choice of the Carleman parameter and comparing with the corresponding upper bound, a contradiction is reached in each case, so that  it is deduced that $u\equiv 0$.

 The assumptions on the decay of the solutions in Theorem \ref{thm:qualitative} are stated in terms of Macdonald's Bessel functions.
 We remark that in the present paper, since we are interested in the role of the size of the mesh $h>0$ and the interpolation phenomenon between the discrete and the continuum regimes, we have to be very careful in the quantification of the arguments; this translates into a fine use of asymptotics of Bessel and hyperbolic functions at different scales.

The Landis-type result for the semidiscrete heat equation is a blueprint for scaled versions of the results in \cite{BFV17, JLMP18}, namely for the discrete time-dependent Schr\"odinger equation in the mesh $(h\Z)^d$. It is expected that the same strategy will go through, although it is likely more technically delicate. We do not tackle this problem in the present paper. On the other hand, the approach used to study the semidiscrete heat equation facilitates the investigation of the elliptic equation, where a similar idea to obtain the corresponding lower bounds is used and, moreover, the estimate of the commutator is automatically obtained from the proof of the Carleman estimate for the semidiscrete heat equation. It is interesting to notice that Theorem \ref{thm:Schr} is indeed a discrete version of the original Landis' conjecture. Observe also that we are not assuming that the potential is either complex-valued or real-valued.
Further, the problem is in connection with the Anderson--Bernoulli model on the lattice, which is the random Schr\"odinger operator on $\ell^2(\Z^d)$ given by
$$
H=-\Delta_{\dis}+\delta V,
$$
where $V$ is a random potential whose values $V_j\in \{0,1\}$ for $j\in \Z^d$ are independent and satisfy $\mathbb{P}[V=0]=\mathbb{P}[V=1]=1/2$, with $\delta>0$ being the strength of the noise, see \cite{BK05,DS20, L22, LZ22}.

\subsection{Quantitative estimates: semidiscrete heat equation}

The proof of Theorem \ref{thm:qualitative} will rely on precise quantitative upper and lower bounds for the solutions to \eqref{p1}. We start showing the upper bound, which holds under the two-time decaying assumption for the solutions.

\begin{thm}[Upper bound, close-to-continuum regime]\label{thm:UpperBound} Let $d\ge 1$ and let $u\in C^1([0,1]:\ell^2((h\Z)^d))$ be a solution to \eqref{p1}.
Assume that for some $\gamma>0$ there exists a finite positive constant $c$ independent of $h$ such that
  \begin{equation}
  \label{eq:inverseG}
 h^d \sum_{j\in \Z^d} \prod_{k=1}^d\Big(\frac{K_{j_{k}}^2(\frac{\gamma}{h^2})}{K_{0}^2(\frac{\gamma}{h^2})}\Big)\big(|u_j(0)|^2  +
|u_j(1)|^2\big)< c.
  \end{equation}
Then there exists $h_0>0$, with $\frac{\gamma}{h_0^2}\ge M=100$, such that for $0<h<h_0$, if we choose $R>1$  with   $Rh< \frac{\gamma}{2}$ and $\frac{R}{h} \ge M$, then
$$
h^d  \int_0^1 \sum_{\substack{j\in \Z^d \\ R-2<|hj|<R+1}}(|u_j(0)|^2+ |u_j(t)|^2) \,dt
  \le
C_{\gamma}e^{-  dR^2/\gamma},
$$
where $C_{\gamma}$ is a positive constant independent of $h$ and $R$.
\end{thm}

Theorem \ref{thm:UpperBound} will follow from a log-convexity argument which will be shown in Section \ref{sec:upper} and hence it may be compared with the convexity result \cite[Theorem 2.2]{FB18}. In the current situation, in order to justify the computations of the formal log-convexity property, the introduction of the scale $h$ makes things more delicate. Indeed, when dealing with a general $h$, the bounds blow-up as the scale shrinks to $0$ and we have to use an approximation argument to get uniform bounds in $h$.

\begin{rmk}
\label{rmk:condition}
The motivation for the condition \eqref{eq:inverseG} in Theorem \ref{thm:UpperBound} (and hence in Theorem \ref{thm:qualitative} $(1)$) relies on the fact that it can be seen as a discrete version of the decay of the inverse continuous Gaussian function. Indeed, this is reasonable in view that the fundamental solution of the semidiscrete heat equation is given in terms of a modified Bessel function $I_{\nu}(x)$ defined in \eqref{eq:Inu}  (see \cite{CGRTV17}) which asymptotically behaves, essentially, as the inverse of a Macdonald's function. On the other hand, it could be thought that a suitable condition might consist of discretizing the continuum weight. Nevertheless, the fundamental solution to the semidiscrete equation is not a mere discretization of the solution to the continuum equation, although the latter converges in certain sense to the former as $h$ shrinks to zero, see Appendix \ref{sub:heats}.
\end{rmk}

Observe that Theorem \ref{thm:UpperBound} is stated for a regime in which $Rh$ is small, so it can be understood as a \textit{close-to-continuum regime}. If we study the corresponding upper bound in the regime in which $Rh$ is large, which can be understood as a \textit{purely discrete regime}, under the same decay condition \eqref{eq:inverseG} we see that the obtained estimate is not enough to conclude a Landis-type result as in Theorem \ref{thm:qualitative}, see Remark \ref{rmk:Rhgrande}. The reason, essentially, is that the relevant parameter $\gamma$ which should be exploited to produce a contradiction is not playing any significant  role in the upper and lower bounds in this regime.

Nevertheless, as it was shown in \cite[Theorem 1.2]{BFV17} in the case $h=1$, it is possible to impose a weaker condition (although not sharp) which produces an appropriate upper bound which gives rise to a contradiction. Motivated by this fact, we state an upper bound in the \textit{purely discrete regime} which will help us prove the Landis-type result in Theorem \ref{thm:qualitative} (2).

\begin{thm}[Upper bound, purely discrete regime]
\label{thm:UpperBound2}
Let $u\in C^1([0,1]:\ell^2((h\Z)^d))$ be a solution to \eqref{p1} and assume
that, for a fixed $\mu>0$,
  \begin{equation}
  \label{eq:inverseG2}
h^d \sum_{j\in \Z^d} \prod_{k=1}^d  \frac{K_{j_k\mu}(\frac{2}{eh^2})}{K_{0}(\frac{2}{eh^2})}\big(|u_j(0)|^2  +
|u_j(1)|^2\big)<\infty.
  \end{equation}
Then there exists $h_0>0$ such that  for $h\in(0,h_0)$ and $R>1$ satisfying $Rh\ge \frac{2}{e\mu}$, the following holds
$$
h^d  \int_0^1 \sum_{\substack{j\in \Z^d \\ R-2<|hj|<R+1}}|u_j(t)|^2 dt
  \le
Ce^{-  \mu c_0\big(\frac{R}{h}\log(Rh)+\frac{ R}{h}\log\mu\big)}
$$
for some positive constants $c_0$  and $C$ independent of $h$ and $R$.
\end{thm}

\begin{rmk}
\label{rmk:example}
It can be checked that the function $u_j(t)=\sum_{j\in \Z^d}\prod_{k=1}^de^{-2t/h^2}\frac{I_{j_k}(2t/h^2+1/h^2)}{I_0(1/h^2)}$ is a solution to equation \eqref{p1} so that \eqref{eq:inverseG} and \eqref{eq:inverseG2} are satisfied, see Section \ref{sec:optimal}.
\end{rmk}

\begin{rmk}
The upper bound in Theorem \ref{thm:UpperBound2} essentially reduces to the one in \cite[Corollary 2.1]{BFV17} when $h=1$. It is possible to adjust the argument of the Macdonald's functions in condition \eqref{eq:inverseG2} to allow the case $h=1$.
\end{rmk}

The change of regimes in $Rh$ exhibited in Theorems \ref{thm:UpperBound} and \ref{thm:UpperBound2} will be also explicit in the lower bound (and actually this change of regimes will become apparent in the proof of the lower bound).

\begin{thm}[Lower bound]
\label{Thm:lowerbound:heat}
Let $u\in C^1([0,1]:\ell^2((h\Z)^d))$ be a nontrivial solution to \eqref{p1}.  Then there exist $h_0>0$, $R_0=R_0(d,u(0),\|V\|_{\infty})>0$, and a positive constant $C$, such that for $R\ge R_0$ and $h\in(0,h_0)$ it follows that
\begin{equation}\label{lowerboundD1}
h^d\int_0^1\sum_{\substack{j\in \Z^d \\ R-2<|hj|<R+1}} (|u_j(0)|^2+|u_j(t)|^2) \,dt\gtrsim
\begin{cases}
e^{-C R^{2}} \qquad &\mbox{ if } Rh\le 1\\
e^{-C\frac{R}{h}\log (Rh)}  \qquad &\mbox{ if } Rh>1.
\end{cases}
\end{equation}
\end{thm}

We emphasize that in both of our quantitative results, the identified critical regime $Rh$ being very small corresponds to the case in which the mesh is shrinking to the continuum Euclidean setting. The estimates in the regime in which $Rh$ is large concern a purely discrete situation and for this we are also giving a uniqueness result in Theorem \ref{thm:qualitative} (2).

\subsection{Quantitative estimates: stationary discrete Schr\"odinger equation}

The proof of Theorem \ref{thm:Schr} will be a consequence of lower bounds for the solutions to \eqref{p1}. The strategy to obtain such lower estimates relies on a Carleman estimate (see Theorem \ref{ThCarlemanElliptic-e}) which is nothing but a simplified version of the Carleman estimate for the semidiscrete heat equation in Theorem \ref{ThCarlemanHeat-h}. Indeed, there are no conditions between the parameters (unlike the case for the semidiscrete heat equation) thanks to the positivity of the commutator.

\begin{thm}[Lower bound]\label{Thm:lowerbound:elliptic}
Let $u\in \ell^2(h\mathbb{Z})^d$ be a non-trivial solution to
$$
\Delta_{\dis} u_j+ V_j u_j=0
$$
for a bounded potential $V$. 
\begin{enumerate}
\item (Case $hR^{1/3}\le1$)  There exists $R_0=R_0(d,\|V\|_{\infty})>0$ and a positive constant $C$ such that for  $h>0$ and $R\ge R_0$ satisfying $hR^{1/3}\le1$, it follows that
\begin{equation}
\label{lowerbound:elliptic0}
e^{-C R^{4/3}} \le h^d\sum_{\substack{j\in \Z^d\\R-2<|hj|<R+1}} |u_j|^2 .
\end{equation}
\item (Case $hR^{1/3}> 1$)  There exists $R_0=R_0(d,\|V\|_{\infty})>0$ and a positive constant $C=C(d)$ such that for  $h>0$ and $R\ge R_0$  satisfying $hR^{1/3}> 1$,
 it follows that $R$ can be written as $R=h^{-\beta}$ for some $\beta>3$, and
\begin{equation*}
e^{-C R^{1+ \frac{1}{\beta}}\log (R^{1-1/\beta})} \le h^d\sum_{\substack{j\in \Z^d\\R-2<|hj|<R+1}} |u_j|^2 .
\end{equation*}
Moreover, if we look at what happens when $R$ is large, not depending on $h$, we get $e^{-C \frac{R}{h}\log (Rh)}$ as a lower bound, where $C=C(d)$.

\end{enumerate}
\end{thm}

\begin{rmk}
The exponent in \eqref{lowerbound:elliptic0} in Theorem \ref{Thm:lowerbound:elliptic} is the same as in the continuum problem and, as shown by Meshkov \cite{Meshkov} and by Bourgain and Kenig \cite{BK05}, it is sharp if the potential is allowed to be complex.
\end{rmk}

\subsection{On sharpness and improvement of the results}

Concerning the sharpness of the results, in Section \ref{sec:optimal} we will show an example that illustrates the optimality of the quantitative estimates in Theorems \ref{thm:UpperBound}, \ref{thm:UpperBound2}, and~\ref{Thm:lowerbound:heat} and hence of Theorem \ref{thm:qualitative}. On the other hand, in view of the results in \cite[Subsection 4.4]{LM18}, and the examples in Subsection \ref{sub:examples} (see Corollary \ref{cor:better}), it is reasonable to expect that $(2)$ in Theorems \ref{thm:Schr} and \ref{Thm:lowerbound:elliptic}, while being analogs of the analysis carried out in the parabolic case to prove Theorem \ref{thm:qualitative}, are not sharp.
At present, it remains uncertain whether $(1)$ in Theorem~\ref{Thm:lowerbound:elliptic}, and consequently in Theorem~\ref{thm:Schr}, particularly concerning the close-to-continuum regime, achieves optimality.
This  issue stands as an interesting problem deserving further exploration.

Furthermore, another interesting open question is whether the two-time decaying condition in Theorems \ref{thm:UpperBound} and  \ref{thm:UpperBound2} can be relaxed to a one-time decaying condition, as it is done in \cite{EKPV-CMP} for parabolic evolutions in the Euclidean case. See \cite{EKPV-JEMS}, where this result is proved for the Schr\"odinger equation in $\R^d$, and also \cite[Theorem 4]{EKPV-JEMS}, where a non-optimal version for the heat equation in the continuum case is shown. Finally, it would be also interesting to extend the results presented in this paper, or improved versions of them, to the non-stationary Schr\"odinger equation in the lattice $(h\Z)^d$.

\subsection{Organization of the paper}

The paper is organised as follows. In Section \ref{sec:prelim} we introduce notational conventions and technical results which will be used later. Section \ref{sec:upper} contains a weighted log-convexity property of the solutions of  \eqref{p1} and the proofs of Theorem \ref{thm:UpperBound} and \ref{thm:UpperBound2}. Section \ref{sec:lower} is devoted to the proof of a Carleman estimate for the semidiscrete heat operator, which will be a fundamental tool in the subsequent proof of Theorem \ref{Thm:lowerbound:heat} in the same section. After that, in Section~\ref{sec:mainproof} we present the proof of the Landis-type results in Theorem \ref{thm:qualitative}. Optimality of the results concerning the semidiscrete heat equation are discussed in Section~\ref{sec:optimal}. Finally, in Section \ref{sec:elliptic}, we present the results within the setting of discrete Schr\"odinger equation, namely the proofs of Theorems~\ref{Thm:lowerbound:elliptic} and~\ref{thm:Schr}.

\medskip
\noindent{\textbf{Acknowledgments}.}
The authors are greatly indebted to Angkana R\"uland for very helpful discussions at various stages of the project.

Part of this work was carried out during L. Roncal and D. Stan's stay at Isaac Newton Institute (INI) for Mathematical Sciences in Cambridge, during the programme Fractional Differential Equations; this work was supported by EPSRC grant no EP/R014604/1.\&\#34. The authors gratefully acknowledge INI's support and hospitality.

\section{Preliminaries and auxiliary results}
\label{sec:prelim}

\subsection{Notational conventions}
\label{sub:notation}

For $f:(h\Z)^d\to \R$, the $\ell^p$ norms on the lattice $(h\Z)^d$ are defined as
$$
\| f\|_{\ell^p((h\Z)^d)} :=\Big( h^d \sum_{j \in \mathbb{Z}^d} |f_j|^p  \Big)^{1/p}, \quad 1\le p<\infty, \qquad \| f\|_{\ell^\infty((h\Z)^d)} :=\sup_{j\in \Z^d}|f_j|,
$$
and the scalar product is given by $  \langle f,g \rangle_{\ell^2((h\Z)^d)} := h^d \sum_{j \in \mathbb{Z}^d} f_j g_j$. We will sometimes use just $\| f\|_{\ell^p}, \|f\|_p, \| f\|_{\ell^\infty},\|f\|_{\infty}$ without further comment. We will also define the left/right difference operators at scale $h$ as
$D_{\pm,k}f_j:=\pm h^{-1}(f_{j\pm e_k}-f_j)$. Observe that the (normalised) discrete Laplacian in \eqref{eq:delta} can be expressed as $\Delta_{\dis}:=\sum_{k=1}^d\Delta_{\dis,k}:=\sum_{k=1}^dD_{-,k}D_{+,k}$. We also write $D_k^sf_j:=\frac{1}{2h}(f_{j+e_k}-f_{j-e_k})$ to denote the symmetric difference operator in the $k$-th direction and $D^sf_j:=\sum_{k=1}^dD_k^sf_j$.
Summation rule has the form
\begin{equation}
\label{eq:sumrules}
\sum_{j\in \Z^d}\sum_{k=1}^d D_{+,k}f_j \, g_j= - \sum_{j\in \Z^d}\sum_{k=1}^df_j \, D_{-,k} g_j.
\end{equation}

With the letters $c,C,\ldots$ we denote structural constants that depend only on the dimension and on parameters that are not relevant. We shall write $X\lesssim Y$ to indicate that $X\le C Y$ with a positive constant $C$ independent of significant quantities and we denote $X\simeq Y$ when simultaneously $X\lesssim Y$ and $Y\lesssim X$. We write $f(x)\sim g(x)$ as $x\to c$ to indicate that $f(x)/g(x)\to 1$ in the limit point $c$. We will sometimes use the notation $X\lll Y$ to denote that $X$ is much smaller than $Y$, i.e., there exists a big constant, say $C\ge 10^3$, such that $X\lll C^{-1}Y$. Analogously, $X\ggg Y$ means that $X\ge C Y$, for instance.

\subsection{Auxiliary results}

In this subsection we collect some technical results which will be used later on.
First, we prove the following energy estimate.

\begin{lem}[Energy estimate]
\label{lem:energy}
Let $u\in C^1([0,1]:\ell^2((h\Z)^d))$ be a solution to the initial value problem \eqref{p1} with $V \in L^\infty((h\mathbb{Z})^d\times[0,1])$ and $u(0) \in \ell^2((h\mathbb{Z})^d)$. Then
\begin{equation}\label{Energy:heat}
\| u(\cdot, t)\|_{\ell^2((h\mathbb{Z})^d)} ^2+2h^d\int_0^t \sum_{j\in \Z^d}\sum_{k=1}^d|D_{-,k}u_j(\tau) |^2\,d\tau \le  e^{ 2 t \|V\|_{L^\infty((h\mathbb{Z})^d\times[0,1]) } }  \| u(\cdot, 0)\|_{\ell^2((h\mathbb{Z})^d)} ^2 , \quad \forall t  \in (0,1].
\end{equation}
\end{lem}
\begin{proof}Let $u$ be a solution to \eqref{p1}. Then, by summation rule \eqref{eq:sumrules},
\begin{align*}
\frac{d}{dt}  \sum_{j\in \Z^d} u^2_j& = -2\sum_{j\in \Z^d}\sum_{k=1}^d|D_{-,k}u_j |^2+2 \sum_{j\in \Z^d} V_j u^2_j\le  -2\sum_{j\in \Z^d}\sum_{k=1}^d|D_{-,k}u_j |^2+ 2\|V\|_{L^\infty((h\mathbb{Z})^d\times[0,1]) } \sum_{j\in \Z^d}  u^2_j.
\end{align*}
We multiply by $e^{ -2 t \|V\|_{L^\infty((h\mathbb{Z})^d\times[0,1]) } } $ so that
$$
\frac{d}{dt}\Big(e^{ -2 t \|V\|_{L^\infty((h\mathbb{Z})^d\times[0,1]) } }  \sum_{j\in \Z^d} u^2_j\Big)\le -2e^{ -2 t \|V\|_{L^\infty((h\mathbb{Z})^d\times[0,1]) } } \sum_{j\in \Z^d}\sum_{k=1}^d|D_{-,k}u_j |^2
$$
and we integrate in time from $0$ to $t$ to obtain
$$
e^{ -2 t \|V\|_{L^\infty((h\mathbb{Z})^d\times[0,1]) } }  \sum_{j\in \Z^d} u^2_j- \sum_{j\in \Z^d} u^2_j(0)\le -2\int_0^te^{ -2 \tau \|V\|_{L^\infty((h\mathbb{Z})^d\times[0,1]) } } \sum_{j\in \Z^d}\sum_{k=1}^d|D_{-,k}u_j(\tau) |^2\,d\tau,
$$
whence the desired result is deduced.
\end{proof}

The second technical result is a smoothing estimate, or Caccioppoli-type inequality for solutions to \eqref{p1}, providing interior parabolic regularity.  We will denote by $A(r_1,r_2)$ the region $\{h j \in  (h\Z)^d : r_1< |hj| <r_2   \} = B_{r_2}\setminus \overline{B_{r_1}}$, with $B_r=B_r(0)\cap (h\Z)^d$.

\begin{lem}[Interior parabolic regularity]
\label{lem:Caccioppoli}

Let $u\in C^1([0,1]:\ell^2((h\Z)^d))$ be a solution to the initial value problem \eqref{p1} with $V: (h\Z)^d\times[0,1] \rightarrow \R$ uniformly bounded in $h$ and $u(0) \in \ell^2((h\mathbb{Z})^d)$.

Then
there exist constants $C_1>0$ and  $C_2>1$ depending on $ \|V\|_{L^{\infty}}$ such that, for $0<t<1$, we have
\begin{multline}\label{Caccioppoli:rings}
C_1\sum\limits_{k=1}^d\int_{0}^{1}\| h^{-1}( u(h(\cdot + e_j),t)- u(h\cdot,t))\|_{\ell^2(A(R-1,R))}^2\,dt
\leq C_2\int_{0}^{1} \| u(h\cdot,t)\|_{\ell^2(A(R-2,R+1))}^2\,dt\\+\|u(\cdot,0)\|^2_{\ell^2(A(R-2,R+1))}.
\end{multline}
\end{lem}
\begin{proof}
Let us denote $A_1:=A(R-1,R)$, $A_2:=A(R-2, R+1)$.
Let $\eta: (h\Z)^d\to [0,1]$ be a cut-off function given by $
\eta=1$ in $A_1$, and $\eta=
0$ in $A_2^c$.
We multiply equation \eqref{p1} by $(u\eta^2)_j(t)$, apply the summation by parts formula, and subsequently integrate it over time
\begin{align}
\label{eqzero}
&\int_{0}^{1}\sum\limits_{j \in  \Z^d}\Big[\partial_tu_j(t) \cdot (u\eta^2)_j(t)+ h^{-2}\sum\limits_{k,n=1}^{d} (u_{j+e_k}(t)- u_j(t))\big((u\eta^2)_{j+e_n}(t) - (u\eta^2)_{j}(t)\big) \\
\notag &\qquad  - V_j(t) u_j(t) (u\eta^2)_j(t) \Big]\,dt=0.
\end{align}
Integrating by parts in the first term above, we have
$$
\int_{0}^{1}\sum_{j\in \Z^d}\partial_tu_j \cdot (u\eta^2)_j\,dt
=-\int_{0}^{1}\sum_{j\in \Z^d}u_j (t) \partial_t(u\eta^2)_j (t)\,dt+\sum_{j\in \Z^d}u_j(1)(u\eta^2)_j(1)-\sum_{j\in \Z^d}u_j(0)(u\eta^2)_j(0).
$$
Moreover, since $\eta$ does not depend on time, we compute directly the above integral
\begin{equation}
\label{inter}
-\int_{0}^{1}\sum_{j\in \Z^d}u_j (t) \partial_t(u\eta^2)_j(t)\,dt=-\frac12 \int_{0}^{1}\sum_{j\in \Z^d}\partial_tu_j^2 (t) \cdot \eta_j^2\,dt =\frac12\sum_{j\in \Z^d}u_j^2(0)\eta_j^2-\frac12\sum_{j\in \Z^d}u_j^2(1)\eta_j^2.
\end{equation}
Hence, \eqref{eqzero} reads as
\begin{align*}
&-\frac12\sum_{j\in \Z^d}u^2_j(0)\eta_j^2+\frac12\sum_{j\in \Z^d}u_j^2(1)\eta_j^2+\int_{0}^{1}\sum_{j\in \Z^d}\Big[ h^{-2}\sum\limits_{k,n=1}^{d} (u_{j+e_k}(t)- u_j(t))((u\eta^2)_{j+e_n}(t) - (u\eta^2)_j(t)) \\
 & \qquad- V_j(t) u_j(t) (u\eta^2)_j(t) \Big]\,dt=0.
\end{align*}
Multiplying by $2$, expanding and rearranging, we have
\begin{align*}
&\sum_{j\in \Z^d}u^2_j(1)\eta_j^2+2\int_{0}^{1}h^{-2}\sum_{j\in \Z^d}\sum\limits_{k,n=1}^{d}(u_{j+e_k}(t)- u_j(t))(u_{j+e_n}(t) - u_j(t))\eta^2_j(t)\,dt\\
&=\sum_{j\in \Z^d}u^2_j(0)\eta_j^2-2\int_{0}^{1}h^{-2}\sum_{j\in \Z^d}\sum\limits_{k,n=1}^{d} (u_{j+e_k}(t)- u_j(t))u_{j+e_n}(t)(\eta_{j+e_n}^2(t)-\eta_j^2(t)) \,dt\\
&\quad+2\int_{0}^{1} h^{-1}\sum_{j\in \Z^d}V_j(t) u_j(t) (u\eta^2)_j(t) \,dt.
\end{align*}
We multiply by $h^d$ and observe that we are boiled down to handling the same terms as in  \cite[Lemma 4.1, (30)]{FBRRS20} with $a_{k,n}=1$. Then, proceeding as in the proof of \cite[Lemma 4.1]{FBRRS20}, we obtain
\begin{align*}
&h^d\sum_{j\in \Z^d}u_j^2(1)\eta_j^2+ \int_{0}^{1}\sum\limits_{k=1}^{d} \|h^{-1}(u(\cdot+he_k,t)-u(\cdot,t))\eta\|^2_{\ell^2((h\Z)^d)}\,dt\\
&\le h^d\sum_{j\in \Z^d}u_j^2(0)\eta_j^2+\frac{1}{4}\int_{0}^{1}\sum\limits_{k=1}^{d}\|h^{-1}(u(\cdot+he_k,t)-u(\cdot,t))\eta\|^2_{\ell^2((h\Z)^d)} \,dt +C_2\int_{0}^{1}\|u(\cdot, t)\|^2_{\ell^2(A_2)}\,dt,
\end{align*}
where $C_2$ depends on $\|V\|_{l^{\infty}(A_2)}$.
From here, we conclude
\begin{multline*}
C_1\int_{0}^{1}\sum\limits_{k=1}^{d} \|h^{-1}(u(\cdot+he_k,t)-u(h\cdot,t))\|^2_{\ell^2(A_1)}\,dt\\
\le h^d\sum_{j\in \Z^d}u_j^2(0)\eta_j^2-h^d\sum_{j\in \Z^d}u_j^2(1)\eta_j^2+C_{2}\int_{0}^{1}\|u(\cdot, t)\|^2_{\ell^2(A_2)}\,dt.
\end{multline*}
as desired.
\end{proof}

\section{Semidiscrete heat equation: upper bounds}
\label{sec:upper}

This section is dedicated to proving the upper bounds in each of the two scenarios: when $Rh$ small and when $Rh$ large. In Subsection \ref{subsec:upperb} we supply the proof of Theorem \ref{thm:UpperBound}. The main ingredient of the proof is  a weighted log-convexity property of the solutions to \eqref{p1}, which we will establish in Subsection \ref{subsection:logconvex}.  Specifically, in Proposition \ref{Prop:logconvexH} we  demonstrate that the quantity
\begin{equation}\label{weighted-energy1}
\sum_{j\in \Z^d} \prod_{k=1}^d \frac{K_{j_{k}}^2(\frac{\gamma}{h^2})}{K_{0}^2 (\frac{\gamma}{h^2})}|u_j(t)|^2
\end{equation}
is logarithmically convex with respect to the time variable. Thus, we can control this quantity by the corresponding weighted norms  for the solution $u$ to \eqref{p1} at times $0$ and $1$, as stated in  \eqref{eq:inverseG}.  Our strategy  is as follows: firstly, we establish the well-defined nature of the quantity $$
H_\delta(t):=\sum_{j\in \Z^d}\prod_{k=1}^d K_{j_{k}}^{2\delta}\big(\frac{\gamma}{h^2}\big) |u_{j}(t)|^2, \quad \text{for }t \in [0,1].
$$
This step relies on weighted energy estimates, detailed in  Proposition \ref{prop:energyK^2} in Subsection \ref{Subsection:weighted-energy}. Subsequently, we prove its log-convexity in Proposition \ref{prop:logconvexKdelta}.  Next, in Proposition~\ref{Prop:logconvexH}  we  conclude that the above energy \eqref{weighted-energy1} is finite by using the upper bound for $H_\delta(t)$ and allowing $\delta \to 1$.  Finally, we employ an abstract convexity result from  \cite{EKPV-JEMS} (which we recall here in Lemma~\ref{lemmaEKPV}) in order to derive the desired log-convexity estimate.

In Subsection \ref{subsec:upperbdiscrete}, we will provide the proof of Theorem \ref{thm:UpperBound2}.

\subsection{Weighted energy estimates}\label{Subsection:weighted-energy}

\begin{prop}\label{prop:energyK^2}
Let $u\in C^1([0,1]:\ell^2((h\Z)^d))$ be a solution to the initial value problem \eqref{p1} with $V\in L^\infty((h\Z)^d\times \R_+)$ and $u(0) \in \ell^2((h\mathbb{Z})^d)$.
Let
\begin{equation*}
H(t) = \sum_{j\in \Z^d} |\omega_j  u_j|^2\quad \text{ with } \quad \omega_j=\omega(hj,t):=\prod_{k=1}^dK_{j_k}\big(\frac{\gamma}{h^2}+\frac{2t}{h^2}\big),
\end{equation*}
for some $\gamma>0$.
 Then
$H(t) \le e^{2t\|V\|_\infty} H(0)$.
\end{prop}
\begin{proof}
We split the proof into two parts. In the first part we assume that all the weighted energy terms are finite and we perform  some formal computations in order to derive the desired inequality. In the second part of the proof we validate the assertions made in the first part by employing a truncation argument.

\noindent \textit{Part I.} Let $f_j:=\omega_ju_j$.  Then $\partial_t f_j=Sf_j +Af_j  + V_j f_j$, where
\begin{equation}
\label{S}
Sf_j=\frac{\partial_t \omega_j}{\omega_j} f_j + \frac{1}{2}\omega_j  \Delta_{\dis} \Big(\frac{f_j}{\omega_j} \Big) +\frac{1}{2}\frac{1}{\omega_j}\, \Delta_{\dis} (f_j\omega_j )
\end{equation}
and
\begin{equation}
\label{A}
Af_j=\frac{1}{2}\omega_j  \Delta_{\dis} \Big(\frac{f_j}{\omega_j} \Big)-\frac{1}{2}\frac{1}{\omega_j} \Delta_{\dis}(f_j\omega_j ) .
\end{equation}
Notice that $S$ is a symmetric operator, while $A$ is an antisymmetric operator.
Then
$$H'(t)=2 \sum_{j\in \Z^d} f_j (f_j)_t = 2 \Ree \sum _{j\in \Z^d} f_j Sf_j +2 \sum _{j\in \Z^d} V_j f_j^2,$$ since $\sum _{j\in \Z^d} f_j  Af_j =0$.
By using \eqref{S}, we can deduce that
\begin{align*}
&H'(t)=2 \sum _{j\in \Z^d} \frac{\partial_t \omega_j}{\omega_j} f_j^2+\sum _{j\in \Z^d} \omega_j  \Delta_{\dis} \Big(\frac{f_j}{\omega_j} \Big) f_j+\sum _{j\in \Z^d} \frac{1}{\omega_j} \Delta_{\dis} (f_j\omega_j ) f_j + 2\sum _{j\in \Z^d}  V_j f_j^2\\
&=2 \sum_{j\in \Z^d} \frac{\partial_t \omega_j}{\omega_j} f_j^2 - \frac{4d}{h^2} \sum_{j\in \Z^d} f_j^2+\frac{2}{h^2} \sum_{j\in \Z^d}\sum_{k=1}^d\frac{\omega_j}{\omega_{j+e_k} } f_{j+e_k} f_j +\frac{2}{h^2} \sum_{j\in \Z^d}\sum_{k=1}^d\frac{\omega_{j+e_k}}{\omega_j} f_{j+e_k}f_j+ 2 \sum_{j\in \Z^d} V_j f_j^2\\
&\le 2 \sum_{j\in \Z^d} \frac{\partial_t \omega_j}{\omega_j} f_j^2 - \frac{4d}{h^2} \sum_{j\in \Z^d} f_j^2+\frac{1}{h^2} \sum_{j\in \Z^d}\sum_{k=1}^d\frac{\omega_{j-e_k}}{\omega_j} f_j^2
\\
& \qquad+\frac{1}{h^2} \sum_{j\in \Z^d}\sum_{k=1}^d\frac{\omega_j}{\omega_{j+e_k}} f_j^2+
\frac{1}{h^2} \sum_{j\in \Z^d}\sum_{k=1}^d\frac{\omega_j}{\omega_{j-e_k}} f_j^2+
\frac{1}{h^2} \sum_{j\in \Z^d}\sum_{k=1}^d\frac{\omega_{j+e_k}}{\omega_j} f_j^2 + 2 \sum_{j\in \Z^d} V_j f_j^2,
\end{align*}
where we employed the inequality $2 f_{j+e_k}f_j \leq f_{j+e_k}^2 + f_j^2$.
Let $\omega_j=\prod_{k=1}^dK_{j_k}\big(\frac{\gamma}{h^2}+\frac{2t}{h^2}\big)$ as in the hypothesis.  Thus,
\begin{align*}
\partial_t \omega_j &=  \frac{2}{h^2} \sum_{k=1}^dK_{j_k}'\big(\frac{\gamma}{h^2}+\frac{2t}{h^2}\big) \prod _{m\ne k, m=1}^d K_{j_m}\big(\frac{\gamma}{h^2}+\frac{2t}{h^2}\big)\\
&=  \frac{2}{h^2} \sum_{k=1}^d \frac{ K_{j_k}'\big(\frac{\gamma}{h^2}+\frac{2t}{h^2}\big)}{K_{j_k}\big(\frac{\gamma}{h^2}+\frac{2t}{h^2}\big)} \prod _{ m=1}^d K_{j_m}\big(\frac{\gamma}{h^2}+\frac{2t}{h^2}\big)=  \frac{2}{h^2} \sum_{k=1}^d \frac{ K_{j_k}'\big(\frac{\gamma}{h^2}+\frac{2t}{h^2}\big)}{K_{j_k}\big(\frac{\gamma}{h^2}+\frac{2t}{h^2}\big)} \omega_j,
\end{align*}
where $K_{j_k}'(z)$ denotes $\partial_zK_{j_k}(z).$
Thus, we can write
$$H'(t) \le\sum_{j\in \Z^d}\sum_{k=1}^d E_{j,k} f_j^2  - \frac{4d}{h^2}H(t)+ 2 \sum_{j\in \Z^d} V_j f_j^2,$$
with
$$
E_{j,k}= \frac{4}{h^2}\frac{ K_{j_k}'}{K_{j_k}} +
\frac{1}{h^2} \frac{K_{j_k-1}}{K_{j_k}} + \frac{1}{h^2} \frac{K_{j_k}}{K_{j_k+1}}+
\frac{1}{h^2} \frac{K_{j_k}}{K_{j_k-1}} +
\frac{1}{h^2} \frac{K_{j_k+1}}{K_{j_k}},
$$
where all the Bessel functions are evaluated at $z=\frac{\gamma}{h^2} +\frac{2t}{h^2}$. We check now that $E_{j,k}\le 0$.
 Using the recurrence formula \eqref{eq:recurrence}, we obtain
\begin{align*}
E_{j,k}&= -\frac{2}{h^2}\frac{  K_{j_k+1}}{K_{j_k}} -\frac{2}{h^2}\frac{  K_{j_k-1}}{K_{j_k}} +
\frac{1}{h^2} \frac{K_{j_k-1}}{K_{j_k}} + \frac{1}{h^2} \frac{K_{j_k}}{K_{j_k+1}}+
\frac{1}{h^2} \frac{K_{j_k}}{K_{j_k-1}} +
\frac{1}{h^2} \frac{K_{j_k+1}}{K_{j_k}}\\
&= \frac{1}{h^2} \Big( - \frac{  K_{j_k+1}}{K_{j_k}} +
 \frac{K_{j_k}}{K_{j_k-1}}  \Big)
+ \frac{1}{h^2} \Big( - \frac{  K_{j_k-1}}{K_{j_k}}  + \frac{K_{j_k}}{K_{j_k+1}}\Big).
\end{align*}
By Tur\'an inequality \eqref{eq:Turan}, we can deduce that $E_{j,k}\le 0$. Thus, we have
 $$
 H'(t) \le \Big(2\|V\|_\infty - \frac{4d}{h^2} \Big) H(t) \le 2\|V\|_\infty   H(t)  \le 2\|V\|_\infty H(t),
 $$
 and integrating the above inequality from $0$ to $t$ for $t>0$, we deduce that the quantity $e^{-2 t \|V\|_\infty} H(t)$ is monotone decreasing in time. Therefore, the conclusion follows.

\noindent \textit{Part II.}  For $m\in\Z$, let us define $\psi^R_m$ by
$$\psi^R_m := \begin{cases}
  K_{m}\big(\frac{\gamma}{h^2}+\frac{2t}{h^2}\big) \mbox{ if } |m| \le R,\\[2mm]
  K_{R}\big(\frac{\gamma}{h^2}+\frac{2t}{h^2}\big) \mbox{ if } |m| \ge R,
\end{cases} $$
and let, for $j\in\Z^d$, $\omega_j^R=\prod_{k=1}^d \psi^R_{j_k}$. Observe that all the inequalities in Part I, up to the choice of the weight $w_j$, are valid in this case and they are finite since we work with the truncated weight $\omega_j^R$, for each $j\in\Z^d$.
We now continue the argument in Part I, but instead of choosing $\omega_j=\prod_{k=1}^dK_{j_k}\big(\frac{\gamma}{h^2}+\frac{2t}{h^2}\big)$, we carry out the computations with the weight $\omega_j^R$ just defined.
Let $
H_R(t)= \sum_{j\in \Z^d} |\omega^R_j  u_j|^2$.
Arguing as above we get
\begin{align*}
H_R'(t) &\le 2 \sum_{j\in \Z^d} \frac{\partial_t \omega^R_j}{\omega^R_j} f_j^2 - \frac{4d}{h^2} \sum_{j\in \Z^d} f_j^2+\frac{1}{h^2} \sum_{j\in \Z^d}\sum_{k=1}^d\frac{\omega^R_{j-e_k}}{\omega^R_j} f_j^2
+\frac{1}{h^2} \sum_{j\in \Z^d}\sum_{k=1}^d\frac{\omega^R_j}{\omega^R_{j+e_k}} f_j^2\\
& \qquad+
\frac{1}{h^2} \sum_{j\in \Z^d}\sum_{k=1}^d\frac{\omega^R_j}{\omega^R_{j-e_k}} f_j^2+
\frac{1}{h^2} \sum_{j\in \Z^d}\sum_{k=1}^d\frac{\omega^R_{j+e_k}}{\omega^R_j} f_j^2 + 2 \sum_{j\in \Z^d} V_j f_j^2\\
& \le- \frac{4d}{h^2}H_R(t)+ 2 \|V\|_\infty   H_R(t)+\sum_{j\in \Z^d}\sum_{k=1}^d E^R_{j,k} f_j^2,
  \end{align*}
  where the quantities $E^R_{j,k}$ enclose the error terms. After tedious computations, it is possible to check that for every $k \in \{1,\dots,d\}$  the error terms are bounded as $ E^R_{j,k}\le \frac{2}{h^2}$, so the conclusion of the proof is true for $H_R$. Finally, we let $R$ tend to $\infty$ to complete the proof.
\end{proof}

\subsection{Log-convexity estimates}  \label{subsection:logconvex}

First, we require a technical result.

\begin{lem}\label{Lemma:K}
Let $\delta \in [0,1]$. For any $x>0$ and any $j\in \Z$, the following holds
\begin{multline}
\label{eq:lambdadelta}
\Lambda_\delta:=\frac{K_{j+1}^{2\delta}(x)}{K_j^{2\delta}(x)} + \frac{K_{j-1}^{2\delta}(x)}{K_j^{2\delta}(x)} -
\frac{K_{j}^{2\delta}(x)}{K_{j-1}^{2\delta}(x)} -\frac{K_{j}^{2\delta}(x)}{K_{j+1}^{2\delta}(x)}\\
+ \frac{ K_{j+1}^{2\delta}(x) }{  K_j^{\delta}(x) K_{j+2}^{\delta} (x)}-
\frac{K_j^{\delta}(x)K_{j+2}^{\delta}(x)}{K_{j+1}^{2\delta}(x)} +
\frac{K_{j-1}^{2\delta}(x)}{K_j^{\delta}(x)K_{j-2}^{\delta}(x)}-
\frac{K_j^{\delta}(x)K_{j-2}^\delta(x)}{K_{j-1}^{2\delta}(x)}  \ge - 2 \Big(1 + \frac{1}{x}+ \frac{1}{4x^3}\Big).
\end{multline}

Moreover, when $\delta =1$, $\Lambda_1 >0$.
\end{lem}
\begin{proof}
Using inequality \eqref{eq:Turan}, we have that
$$
\frac{K_{j+1}^{2\delta}}{K_j^{2\delta}} \ge \frac{K_{j}^{2\delta}}{K_{j-1}^{2\delta}} \quad \text{and} \quad \frac{K_{j-1}^{2\delta}}{K_j^{2\delta}} \ge \frac{K_{j}^{2\delta}}{K_{j+1}^{2\delta}}.
$$
For the remaining terms, we combine the bounds \eqref{Turan2} and \eqref{Turan3a} to conclude that $\displaystyle \frac{K_{j+1}^2(x)} {K_j(x)K_{j+2}(x)} \ge \frac{1}{1+\frac{1}{x}+\frac{1}{4x^3}}$ for all $j\in \mathbb{Z}.$
Thus, by estimating from below the negative terms in the second line of  \eqref{eq:lambdadelta}, we obtain that
$$-
\frac{K_j^{\delta}(x)K_{j+2}^{\delta}(x)}{K_{j+1}^{2\delta}(x)} -
\frac{K_j^{\delta}(x)K_{j-2}^\delta(x)}{K_{j-1}^{2\delta}(x)} \ge -2 \Big( 1+ \frac{1}{x} +\frac{1}{4x^3}\Big)^\delta \ge -2 \Big(1 + \frac{1}{x}+\frac{1}{4x^3}\Big),
$$
for any $x>0$ and  $j\in \mathbb{Z}$.

When $\delta=1$, it was proved in \cite[p. 269]{FB18} that $\Lambda_1>0$.
\end{proof}

It will be convenient to recall here part of \cite[Lemma 2]{EKPV-JEMS} which relates the log-convexity of the $L^2$-norm of a function with a weak ``pseudo-positivity'' condition on the commutator of symmetric and antisymmetric part of an operator.

\begin{lem}[{\cite[Lemma 2]{EKPV-JEMS}}]\label{lemmaEKPV}
Let $S$ be a symmetric operator and let $A$ be an antisymmetric operator, both allowed to depend on the time variable. Let $G$ be a positive function, $f(x,t)$ be a reasonable function,
$$
H(t)=\langle f,f\rangle, \quad D(t)=\langle Sf,f\rangle, \quad \partial_tS=S_t.
$$
Then, if
$$
|\partial_t f-Af-Sf|\le M_1|f|+G \quad \text{ in }  \R^n\times [0,1], \quad S_t+[S,A]\ge -M_0,
$$
and $
M_2=\sup_{[0,1]}\|G(t)\|/\|f(t)\|$
is finite, then $\log H(t)$ is  ``convex'' in $[0,1]$, in the sense that there is a universal constant $N$ such that
$$
H(t)\le e^{N(M_0+M_1+M_2+M_1^2+M_2^2)}H(0)^{1-t}H(1)^t, \text{ when } 0\le t\le 1.
$$
\end{lem}

\begin{prop}\label{prop:logconvexKdelta}
Let $\delta\in (0,1),\, h ,\gamma>0 $  and define
$$
H_\delta(t):=\sum_{j\in \Z^d}\prod_{k=1}^d K_{j_{k}}^{2\delta}\big(\frac{\gamma}{h^2}\big) |u_{j}(t)|^2, \quad \text{for }t \in [0,1].
$$
Then, $H_\delta$ is log-convex, in the sense there exists a constant $N_1>0$ such that
$$
H_\delta(t) \le e^{N_1( \|V\|_\infty +\frac{4d}{h^4} (1+ \frac{h^2}{\gamma}+\frac{h^6}{4\gamma^3}) ) } H_\delta(0)^{1-t} H_\delta(1)^t,
$$
for every $t \in [0,1]$.
\end{prop}
\begin{proof}
First of all, we need to justify  that $H_\delta (t)$ is well defined. This follows using Proposition \ref{prop:energyK^2} and the following claim
\begin{equation}\label{ineq:Bessel}
K_{j}^\delta \Big(\frac{\gamma}{h^2}\Big)< cK_j\Big(\frac{\gamma}{h^2} + \frac{2t}{h^2}\Big)  \quad \text{for all }t\in [0,1],\, \,j\in \mathbb{Z},
\end{equation}
for some positive $c$.
Let us prove the claim. First, we argue for large $j$. Let $t=1$.
We utilize the asymptotic approximation for large order \eqref{eq:asympLargeOrd}. Therefore, there exists a $j_0 \in \mathbb{Z}$ such that, for $j \ge j_0$, the following estimate holds
$$
\frac{ K_j(\frac{\gamma}{h^2} + \frac{2}{h^2})}{K_j^\delta (\frac{\gamma}{h^2})} \sim  j^{j(1-\delta)} j^{ - \frac{1}{2}(1-\delta)} 2 ^{j(1-\delta)} e^{-j(1-\delta)}  \bigg(\frac{\big(\frac{\gamma}{h^2} \big)^\delta}{\frac{\gamma}{h^2} + \frac{2}{h^2}}\bigg)^j.
$$

First, we will prove that this quantity is larger than $1$ for large $j.$   The dominant term is the first one, $ j^{j(1-\delta)}$, and we will use it to control all the small terms.
Indeed, we have the lower bound
$$
\frac{ K_j(\frac{\gamma}{h^2} + \frac{2}{h^2})}{K_j^\delta (\frac{\gamma}{h^2})} \ge j^{\frac{1}{4}j(1-\delta)}  j^{(\frac{1}{4}j- \frac{1}{2})(1-\delta)}  \Big( 2 \frac{j^{\frac{1}{4}}}{e} \Big) ^{j(1-\delta)}
\Bigg( \frac{  j^{\frac{1}{4}(1-\delta)} }{  \frac{\frac{\gamma}{h^2} + \frac{2}{h^2}}{(\frac{\gamma}{h^2} )^\delta} } \Bigg)^j.
$$
Let $j_1$ be such that $
j_1^{\frac{1}{4}(1-\delta)}\ge  \frac{\frac{\gamma}{h^2} + \frac{2}{h^2}}{(\frac{\gamma}{h^2} )^\delta}$.
For $j \ge \max \{j_0, 2, e^{4},j_1\}$ we have that $\frac{ K_j(\frac{\gamma}{h^2} + \frac{2}{h^2})}{K_j^\delta (\frac{\gamma}{h^2})} \ge j^{\frac{1}{4}j(1-\delta)}>1$.
Now, we argue for smaller values of $j$.  For $j\in (0, \max \{j_0, 2, e^{4},j_1 \}) \cap \mathbb{Z}$, by the positivity of $K_j(x)$, there exists $b>0$ such that $
\frac{ K_j(\frac{\gamma}{h^2} + \frac{2}{h^2})}{K_j^\delta (\frac{\gamma}{h^2})}  \ge b$.
Thus, for all $j\in \mathbb{Z}$ we have that $
\frac{ K_j(\frac{\gamma}{h^2} + \frac{2}{h^2})}{K_j^\delta (\frac{\gamma}{h^2})}  \ge \min\{1,b\}$.
Since $K_j(\cdot)$ is monotone decreasing, we conclude that
$$
\frac{ K_j(\frac{\gamma}{h^2} + \frac{2t}{h^2})}{K_j^\delta (\frac{\gamma}{h^2})}  \ge \min\{1,b\}, \quad \text{for all } t\in [0,1],\,\, j\in \mathbb{Z}.
$$
Finally, by taking $\displaystyle c= \frac{1}{\min\{1,b\}}$, we conclude the proof of inequality \eqref{ineq:Bessel}.

We continue the proof of Proposition \ref{prop:logconvexKdelta}.  We will make use of \cite[Lemma 2]{EKPV-JEMS} (that we reproduced in Lemma \ref{lemmaEKPV} above). Let $f_j:=\omega_ju_j$, with
\begin{equation}
\label{eq:omegajdelta}
\omega_j=  \prod_{k=1}^d \omega_{j_k} \quad  \text{with } \omega_{j_k}:= K_{j_k}^\delta\big(\frac{\gamma}{h^2}\big).
\end{equation}
Then $\partial_tf_j=Sf_j+Af_j+V_j f_j$ where $S$ and $A$ are defined in \eqref{S} and \eqref{A}. Notice that the weight $\omega$ in this situation does not depend on time.

Moreover, since the weight is given by a tensorial product, we can rewrite  $S$ and $A$ in an equivalent form as $S=\sum_{k=1}^d S_k$ and $A=\sum_{k=1}^d A_k$,
with
\begin{equation*}
S_k f_j=\frac{1}{2} \omega_{j_k}  \Delta_{\dis,k} \Big(\frac{f_j}{ \omega_{j_k}} \Big) +\frac{1}{2}\frac{1}{\omega_{j_k}}\, \Delta_{\dis,k} (f_j \omega_{j_k} ), \qquad
A_kf_j=\frac{1}{2}\omega_{j_k}  \Delta_{\dis,k} \Big(\frac{f_j}{ \omega_{j_k}} \Big)-\frac{1}{2}\frac{1}{ \omega_{j_k}} \Delta_{\dis,k}(f_j \omega_{j_k} ) .
\end{equation*}
After a tedious computation, we obtain that $[S_k,A_m]f_j=0$ for all $k\ne m$, and all $j\in \mathbb{Z}^d.$
Thus, the commutator's expression reduces to $[S,A]f_j=\sum_{k=1}^d [S_k,A_k]f_j$.
 We verify that the conditions of \cite[Lemma 2]{EKPV-JEMS} are satisfied, namely:
$|\partial_t f_j - S f_j - A f_j| \leq M_1 |f_j|$ and $\langle [S,A] f_j, f_j \rangle \geq -M_0 \|f_j\|_2^2$ (note that the weight $w_j$ does not depend on the variable $t$, so the action of $S_t$ is irrelevant here) for some constants $M_0, M_1 > 0$.
  The first estimate can be easily derived since the potential $V_j$ is bounded, namely $
|\partial_tf_j - Sf_j-Af_j |= |V_j f_j|\le \|V_j\|_\infty |f_j|$.

For the second estimate, we notice that
\begin{align*}
 \sum_{j\in \Z^d}\langle[S,A]f_j,f_j\rangle &= \sum_{k=1}^d \sum_{j\in \Z^d} \langle[S_k,A_k]f_j,f_j\rangle =  2 \Ree   \sum_{k=1}^d \sum_{j\in \Z^d}\langle S_kf_j, A_kf_j\rangle
  \\
  &= \frac{1}{2} \sum_{k=1}^d \Bigg[ \sum_{j\in \Z^d} \Big( \omega_{j_k} \Delta_{\dis,k} \Big(\frac{f_j}{\omega_{j_k}} \Big) \Big)^2 -
\sum_{j\in \Z^d} \Big(\frac{1}{ \omega_{j_k}} \Delta_{\dis,k} \big(f_j\omega_{j_k} \big) \Big)^2 \Bigg]=:\frac{1}{2}\sum_{k=1}^d \Lambda_k.
\end{align*}
We prove that each $\Lambda_k$ is bounded from below by $-c \|f_j\|_2^2$ for some $c>0$. Indeed, we have that
\begin{align*}
\Lambda_k &= \frac{1}{h^4} \sum_{j\in \Z^d} \Big(  \frac{ \omega_{j_k-e_k}\omega_{j_k+e_k}}{\omega^2_{j_k}} -  \frac{\omega^2_{j_k}}{ \omega_{j_k-e_k}\omega_{{j_k}+e_k}}  \Big) ( f_{j+e_k} -f_{j_-e_k})^2+
\frac{1}{h^4}\sum_{j\in \Z^d}\sum_{k=1}^d  \Big( \frac{\omega^2_{j_k-e_k}}{\omega^2_{j_k}}+ \frac{\omega^2_{j_k+e_k}}{\omega^2_{j_k}}\\
&\qquad-\frac{\omega^2_{j_k}}{\omega^2_{j_k-e_k}} - \frac{\omega^2_{j_k}}{\omega^2_{j_k+e_k}}
+ \frac{\omega^2_{j_k-e_k}}{\omega_{j_k} \omega_{j_k-2e_k}} - \frac{\omega_{j_k}\omega_{j_k+2e_k}}{\omega^2_{j_k+e_k}} -
 \frac{\omega_{j_k} \omega_{j_k-2e_k}}{\omega^2_{j_k-e_k}} + \frac{\omega^2_{j_k+e_k}}{\omega_{j_k}\omega_{j_k+2e_k}} \Big) f^2_j.
\end{align*}
Now we take into account the definition of the weight $\omega_j$ in \eqref{eq:omegajdelta}.
The first sum is positive since, by the Tur\'an inequality \eqref{eq:Turan}, the following coefficient is positive, namely $
  \frac{ \omega_{j_k-e_k}\omega_{j_k+e_k}}{\omega^2_{j_k}} -  \frac{\omega^2_{j_k}}{ \omega_{{j_k}-e_k}\omega_{{j_k}+e_k}}  >0$.
The coefficient in the second sum can be rewritten in terms of the Macdonald's function $K_{{j_k}}$, and, by Lemma \ref{Lemma:K}, this quantity is bounded from below by $-2 (1+ \frac{h^2}{\gamma}+\frac{h^6}{4\gamma^3}).$
Thus
$$
\langle[S,A]f,f \rangle =\sum_{k=1}^d\sum_{j\in \Z^d}\langle[S_k,A_k]f_j,f_j\rangle\ge -\frac{4d}{h^4} \Big(1+ \frac{h^2}{\gamma}+\frac{h^6}{4\gamma^3}\Big) \|f\|_2^2.$$
Thus, by \cite[Lemma 2]{EKPV-JEMS}, there is a universal constant $N_1$ such that
$$
H_\delta(t) \le e^{N_1( \|V\|_\infty +\frac{4 d }{h^4} (1+ \frac{h^2}{\gamma}+\frac{h^6}{4\gamma^3}) ) } H_\delta(0)^{1-t} H_\delta(1)^t.
$$
The proof is complete.
\end{proof}

The fact that the lower bound has a bad dependence on $h$ in Proposition \ref{prop:logconvexKdelta} (it blows up as $h\to0$) can be improved. Indeed, the latter proposition can be used to justify the previous computations in the case $\delta=1$. In \cite[Theorem 2.1]{FB18}, the positivity of the commutator (which is the same for both heat and Schr\"odinger evolutions) is proved, and therefore we have the following

\begin{prop}\label{Prop:logconvexH} Let $u$ be a solution to \eqref{p1}.  Assume that
 \begin{equation}\label{hipotesis}
  h^d \sum_{j\in \Z^d} \prod_{k=1}^d\Big(\frac{K_{j_{k}}^2(\frac{\gamma}{h^2})}{K_{0}^2(\frac{\gamma}{h^2})}\Big)\big(u^2_j(0)  +
u_j^2(1)\big) < c
 \end{equation}
 for a positive constant independent of $h$.
Then, there exists a constant $N_2>0$ such that, for every $t\in (0,1)$, the following log-convexity estimate holds
\begin{multline*}
 h^d \sum_{j\in \Z^d} \prod_{k=1}^d\Big(\frac{K_{j_{k}}^2(\frac{\gamma}{h^2})}{K_{0}^2(\frac{\gamma}{h^2})}\Big)  u^2_j (t)\\
 \le e^{N_2\|V\|_\infty } h^{d} \Big[\sum_{j\in \Z^d} \prod_{k=1}^d\Big(\frac{K_{j_{k}}^2(\frac{\gamma}{h^2})}{K_{0}^2(\frac{\gamma}{h^2})}\Big)  u^2_j(0)  \Big]^{1-t}\Big[\sum_{j\in \Z^d} \prod_{k=1}^d\Big(\frac{K_{j_{k}}^2(\frac{\gamma}{h^2})}{K_{0}^2(\frac{\gamma}{h^2})}\Big)  u_j^2(1)\Big]^t.
 \end{multline*}
\end{prop}
\begin{proof}
Let $
H(t):=\sum_{j\in \Z^d} \prod_{k=1}^d \frac{K_{j_{k}}^2(\frac{\gamma}{h^2})}{K_{0}^2 (\frac{\gamma}{h^2})}u_j^2(t)$.
By Fatou's Lemma, Proposition \ref{prop:logconvexKdelta} and Dominated Convergence Theorem, we have that
\begin{align*}
&\sum_{j\in \Z^d} \prod_{k=1}^dK_{j_{k}}^2\big(\frac{\gamma}{h^2}\big) u^2_j= \sum_{j\in \Z^d} \prod_{k=1}^d\lim_{\delta \to 1} K_{j_{k}}^{2\delta}\big(\frac{\gamma}{h^2}\big) u^2_j \le\lim_{\delta \to 1}\sum_{j\in \Z^d} \prod_{k=1}^d K_{j_{k}}^{2\delta}\big(\frac{\gamma}{h^2}\big) u^2_j  \\
&\quad \le  \lim_{\delta \to 1} e^{N( \|V\|_\infty +\frac{4d}{h^4} (1+ \frac{h^2}{\gamma} +\frac{h^6}{4\gamma^3} )  ) }
\Big[\sum_{j\in \Z^d} \prod_{k=1}^dK_{j_{k}}^{2\delta}\big(\frac{\gamma}{h^2}\big)  u^2_j(0)  \Big]^{1-t}\Big[\sum_{j\in \Z^d} \prod_{k=1}^dK_{j_{k}}^{2\delta}\big(\frac{\gamma}{h^2}\big) u_j^2(1)\Big]^t\\
&\quad =e^{N( \|V\|_\infty +\frac{4d}{h^4} (1+ \frac{h^2}{\gamma}+\frac{h^6}{4\gamma^3} ) ) }
\Big[\sum_{j\in \Z^d} \prod_{k=1}^dK_{j_{k}}^{2}\big(\frac{\gamma}{h^2}\big)  u^2_j(0)  \Big]^{1-t}\Big[\sum_{j\in \Z^d} \prod_{k=1}^dK_{j_{k}}^{2}\big(\frac{\gamma}{h^2}\big) u_j^2(1)\Big]^t.
\end{align*}
Thus, by virtue of \eqref{hipotesis},  we proved that $\sum_{j\in \Z^d} \prod_{k=1}^dK_{j_{k}}^2(\frac{\gamma}{h^2}) u^2_j$ is finite for all $0\le t\le 1$.
This allows us to prove log-convexity for the quantity $H(t)$.
We follow the same approach as in the proof of Proposition \ref{prop:logconvexKdelta} and utilize the fact that $\Lambda_1>0$. Here, $\Lambda_1$ is defined in equation \eqref{eq:lambdadelta}. Thus, in this case, we have $\langle[S,A]f,f \rangle \ge 0$. Furthermore, by employing Lemma \ref{lemmaEKPV} once again, we can establish the existence of a universal constant $N_2$ such that the following inequality holds
$$
H(t) \le e^{N_2 \|V\|_\infty  } H(0)^{1-t} H(1)^t \quad \text{for all }t\in [0,1].
$$
With this, we conclude the proof.
\end{proof}

\subsection{Upper bound: proof of Theorem \ref{thm:UpperBound} (case $Rh$ small).}
\label{subsec:upperb}

Observe that Proposition~\ref{Prop:logconvexH} guarantees that for all  $t\in [0,1]$, the quantity
  $  h^d \sum_{j \in \mathbb{Z}^d} \prod_{k=1}^d \big(\frac{K_{j_k}^2(\frac{\gamma}{h^2})}{K_0^2(\frac{\gamma}{h^2})}\big) |u_j(t)|^2$ is finite.
 Consequently,
  \begin{align*}
 h^d \int_0^1  \sum_{\substack{j\in \Z^d \\ R-2<|jh|<R+1}} |u_j(t)|^2\,  dt &= h^d\int_0^1 \sum_{\substack{j\in \Z^d \\ R-2<|jh|<R+1}} \prod_{k=1}^d\Big(\frac{K_0^2(\frac{\gamma}{h^2})}{K_{j_k}^2(\frac{\gamma}{h^2})} \frac{K_{j_k}^2(\frac{\gamma}{h^2})}{K_0^2(\frac{\gamma}{h^2})}\Big) |u_j(t)|^2\, dt \\
  & \le h^d \sup_{\substack{j\in \Z^d \\ R-2<|jh|<R+1}} \prod_{k=1}^d\Big( \frac{K_0^2(\frac{\gamma}{h^2})}{K_{j_k}^2(\frac{\gamma}{h^2})} \Big)\, \sup_{t\in[0,1]}  \sum_{j \in \mathbb{Z}^d} \prod_{k=1}^d\Big( \frac{K_{j_k}^2(\frac{\gamma}{h^2})}{K_0^2(\frac{\gamma}{h^2})}\Big) |u_j(t)|^2,
\end{align*}
and, by virtue of Proposition \ref{Prop:logconvexH}, the quantity with the $\sup_{t\in [0,1]}$ can be bounded by a finite constant $C$ uniformly in $h$.
Hence, our goal is to provide an upper bound estimation for the following quantity
\begin{equation}
\label{eq:cantidadd}
\sup_{\substack{j\in \Z^d \\ R-2<|jh|<R+1}} \prod_{k=1}^d\Big( \frac{K_0^2(\frac{\gamma}{h^2})}{K_{j_k}^2(\frac{\gamma}{h^2})} \Big).
\end{equation}

To facilitate the analysis, let us first consider the one-dimensional case. It suffices to examine nonnegative values of $j\in \Z$, as per the observation in \eqref{eq:par}. We will consider $jh\simeq R$ and $K_j^2(\frac{\gamma}{h^2})=K_j^2(j\frac{\gamma}{jh^2})$ (which implies that $\gamma/(jh^2)\simeq \gamma/(Rh)$ is fixed). Then, we will use   the asymptotic relations \eqref{eq:K0large} for large argument and \eqref{eq:Kjlarge} for large index (we bring them here for the sake of the reading)
$$K_0(z)\le C_M  \frac{1}{\sqrt{z}}e^{-z}, \quad \text{ for all } z\ge M,$$
and
$$K_j(jz) \ge c_M  \frac{1}{\sqrt{j}(1+z^2)^{1/4}}e^{-j\big(\sqrt{1+z^2}+\log\frac{z}{1+\sqrt{1+z^2}}\big)}, \quad \text{ for all } |j| \ge M,$$
where $M=100$ has been fixed in the hypotheses.
Thus, for $h$ as in the hypothesis, which ensures that $\frac{\gamma}{h^2} \ge M$ and  $j\simeq \frac{R}{h}  \ge M$,   we have
\begin{align}
\label{eq:cantidad}
\notag\frac{K_0^2(\frac{\gamma}{h^2})}{K_j^2(\frac{\gamma}{h^2})}&\lesssim \frac{jh^2\sqrt{1+\frac{\gamma^2}{j^2h^4}}}{\gamma}\exp\Bigg[-2\Bigg(\frac{\gamma}{h^2}-j\sqrt{1+\frac{\gamma^2}{j^2h^4}}-j\log\Big(\frac{\frac{\gamma}{jh^2}}{1+\sqrt{1+\frac{\gamma^2}{j^2h^4}}}\Big)\Bigg)\Bigg]\\
&\simeq\frac{Rh}{\gamma}\sqrt{1+\frac{\gamma^2}{R^2h^2}}\exp\Bigg[-2\Bigg(\frac{\gamma}{h^2}-\frac{R}{h}\sqrt{1+\frac{\gamma^2}{R^2h^2}}-\frac{R}{h}\log\Big(\frac{\frac{\gamma}{Rh}}{1+\sqrt{1+\frac{\gamma^2}{R^2h^2}}}\Big)\Bigg)\Bigg].
\end{align}
We rewrite the estimate \eqref{eq:cantidad} in the form
$$
\frac{K_0^2(\frac{\gamma}{h^2})}{K_j^2(\frac{\gamma}{h^2})}\lesssim\bigg(\sqrt{\frac{R^2h^2}{\gamma^2}+1}\bigg)\exp\Bigg[-2\Bigg(\frac{\gamma}{h^2}-\frac{\gamma}{h^2}\bigg(\sqrt{\frac{R^2h^2}{\gamma^2}+1}\bigg)+\frac{R}{h}\log\Big(\frac{Rh+\gamma\sqrt{\frac{R^2h^2}{\gamma^2}+1}}{\gamma}-1+1\Big)\Bigg)\Bigg]
.$$
We will use the Taylor expansion $\sqrt{1+z^2}\sim 1+z^2/2$ and $\log(1+z)\sim z$ as $|z|<1$. In order to do this, we recall that, by hypothesis $\frac{Rh}{\gamma}< \frac{1}{2}$, thus
$\sqrt{\frac{R^2h^2}{\gamma^2}+1} \sim \frac{R^2h^2}{2\gamma^2}+1$ and
$$\frac{Rh+\gamma\sqrt{\frac{R^2h^2}{\gamma^2}+1}}{\gamma}-1 \sim \frac{Rh+\gamma(\frac{R^2h^2}{2\gamma^2}+1)}{\gamma}-1=\frac{Rh}{\gamma}+\frac{R^2h^2}{2\gamma^2}$$
which has modulus less than $1$, using again the fact that $\frac{Rh}{\gamma}<\frac12$. Then,
\begin{align*}
\frac{K_0^2(\frac{\gamma}{h^2})}{K_j^2(\frac{\gamma}{h^2})}& \lesssim\bigg(1 + \frac{R^2h^2}{2\gamma^2}\bigg)\exp\Big[-2\frac{\gamma}{h^2}+2\frac{\gamma}{h^2}\bigg(\frac{R^2h^2}{2\gamma^2}+1\bigg)-2\frac{R}{h}\cdot
\frac{1}{\gamma} \left(Rh + \frac{R^2h^2}{2\gamma} \right)\Big]\\
&\lesssim\Big(1+\frac{1}{2\gamma^2}\Big)\exp\Big[-\frac{R^2}{\gamma}\Big(1+\frac{Rh}{\gamma}\Big)\Big]\le C_{\gamma}e^{-R^2/\gamma}.
\end{align*}

  We extend this estimate to the multidimensional case as follows.
Let us consider the tensorial product \eqref{eq:cantidadd} for some fixed $j\in\mathbb{Z}^d$ such that $R-2 < |jh|<R+1$.  Then, there exists at least one component $j_k$ such that $ |j_kh| \sim R$, let us assume this is $j_1$.
 We split the product in
\begin{align*}
   \frac{K_0^2(\frac{\gamma}{h^2})}{K_{j_1}^2(\frac{\gamma}{h^2})} \prod_{k=2}^d\Big( \frac{K_0^2(\frac{\gamma}{h^2})}{K_{j_k}^2(\frac{\gamma}{h^2})} \Big) \lesssim C_{\gamma}e^{-R^2/\gamma} \cdot 1,
\end{align*}
where the quotients in the tensorial product are simply bounded by $1$, since the modified Bessel functions are monotone increasing with respect to order (see \eqref{monotonicityBessel}).
  The proof is complete now.

\begin{rmk}
\label{rmk:Rhgrande}
Let $d\ge 1$ and $u\in C^1([0,1]:\ell^2((h\Z)^d))$ be a solution to \eqref{p1}.
Assume that for some $\gamma>0$ there exists a finite positive constant $c(\gamma)$ independent of $h$ such that
  \begin{equation*}
 h^d \sum_{j\in \Z^d} \prod_{k=1}^d\Big(\frac{K_{j_{k}}^2(\frac{\gamma}{h^2})}{K_{0}^2(\frac{\gamma}{h^2})}\Big)\big(|u_j(0)|^2  +
|u_j(1)|^2\big)< c(\gamma)
  \end{equation*}
  and that there exists $h_0>0$ with $\frac{\gamma}{h_0^2}\ge M=100$.
If now we choose $R>1$  such that   $\frac{Rh}{\gamma}>1 $ and $\frac{R}{h} \ge M$, for $0<h<h_0$, then arguing as in the previous proof,  the quantity in \eqref{eq:cantidad} can be estimated by
\begin{align*}
&\frac{Rh}{\gamma}\Big(1+\frac{\gamma^2}{2R^2h^2}\Big)\exp\Big[-\frac{2\gamma}{h^2}+\frac{2R}{h}\Big(1+\frac{\gamma^2}{R^2h^2}\Big)-\frac{2R}{h}\log(Rh)-\frac{2R}{h}\log\Big(\frac{2}{\gamma}\Big(1+\frac{\gamma^2}{4R^2h^2}\Big)\Big)\Big]\\
&\qquad \qquad \lesssim \frac{1}{\gamma}\exp\Big[c\Big(-\frac{2R}{h}\log(Rh)-\frac{2R}{h}\Big(\log\Big(\frac{2}{\gamma}\Big)-1\Big)+\log (Rh)\Big)\Big],
\end{align*}
and observe that the leading term on the right hand side is $e^{-c\frac{R}{h}\log(Rh)}$, $c$ being independent of all relevant parameters. Notice that the term $Rh$ in front of the exponential on the left hand side has been absorbed by the exponential, with the cost of certain constant. On the other hand, the term $1/\gamma$ in front of the exponential is not playing a relevant role. Eventually, this estimate does not produce a contradiction if we combine it with the lower bound in Theorem \ref{Thm:lowerbound:heat}. This motivates to consider a different condition for the regime $Rh$ large.

\end{rmk}

\subsection{Upper bound: proof of Theorem \ref{thm:UpperBound2} (case $Rh$ large)}
\label{subsec:upperbdiscrete}

Analogously to the approach in \cite[Lemma 3.1 and Remark 3.2]{FB18} (see also \cite[Lemma 2.2]{BFV17}), we can prove that
$$
h^d\sum_{j\in \Z^d}e^{\beta\cdot j}|u_j(t)|^2
\le e^{C\|V\|_{\infty}}h^d\sum_{j\in \Z^d}e^{\beta\cdot j}\big(|u_j(0)|^2+|u_j(1)|^2\big).
$$
In fact, it can be verified that the commutator computation in the proof of Proposition \ref{prop:logconvexKdelta} is identical to that in \cite[proof of Theorem 2.2]{FB18}. This commutator computation, with the specific weight $e^{\beta\cdot j}$ delivers the result in  \cite[Lemma 3.1 and Remark 3.2]{FB18}, scaled in $h$.

If we multiply by $\prod_{k=1}^de^{-\frac{2}{eh^2}\cosh{(\beta_k/\mu)}}$ and integrate in $\beta\in \R^d$ we have, in virtue of \eqref{eq:RepK},
\begin{align*}
&\int_{\R^d}e^{\beta\cdot j}\prod_{k=1}^de^{-\frac{2}{eh^2}\cosh{(\beta_k/\mu)}}\,d\beta=2^d\int_{0}^{\infty}\cdots\int_{0}^{\infty}\prod_{k=1}^d\frac{e^{\beta_k j_k}+e^{-\beta_k j_k}}{2}e^{-\frac{2}{eh^2}\cosh{(\beta_k/\mu)}}\,d\beta_1\cdots\beta_d\\
&\quad=2^d\int_{0}^{\infty}\cdots\int_{0}^{\infty}\prod_{k=1}^d \cosh(\beta_k j_k)e^{-\frac{2}{eh^2}\cosh{(\beta_k/\mu)}}\,d\beta_1\cdots\beta_d=(2\mu)^d \prod_{k=1}^d K_{j_k\mu}\Big(\frac{2}{eh^2}\Big).
\end{align*}
 By Tonelli, we deduce that
$$
 h^d \sum_{j\in \Z^d} \prod_{k=1}^d K_{j_k\mu}\Big(\frac{2}{eh^2}\Big)|u_j(t)|^2\le e^{C\|V\|_{\infty}}
 h^d \sum_{j\in \Z^d} \prod_{k=1}^d K_{j_k\mu}\Big(\frac{2}{eh^2}\Big)\big(|u_j(0)|^2  +
|u_j(1)|^2\big)
$$
and from here trivially
\begin{equation}\label{eq:labuena}
 h^d \sum_{j\in \Z^d} \prod_{k=1}^d \frac{K_{j_k\mu}(\frac{2}{eh^2})}{K_{0}(\frac{2}{eh^2})}|u_j(t)|^2\le e^{C\|V\|_{\infty}}
 h^d \sum_{j\in \Z^d} \prod_{k=1}^d  \frac{K_{j_k\mu}(\frac{2}{eh^2})}{K_{0}(\frac{2}{eh^2})})\big(|u_j(0)|^2  +
|u_j(1)|^2\big),
\end{equation}
which is bounded by a constant $C$ uniformly in $h$, by hypothesis.

On the one hand, taking $h$ small enough, we can use the asymptotics \eqref{eq:K0large} so that
\begin{equation}
\label{eq:asyK0}
K_{0}\Big(\frac{2}{eh^2}\Big)\sim h\exp\big[-\frac{c}{h^2}\big].
\end{equation}
On the other hand, taking into account that $|jh|\simeq R$ let us assume, without loss of generality, that $|j_1h|\simeq R$. By using the asymptotics in \eqref{eq:asympLargeOrd} we obtain, for $Rh>\frac{2}{e\mu}$,
\begin{equation}
\label{eq:estimateKjm}
 2\mu K_{j_1\mu}\Big(\frac{2}{eh^2}\Big) \sim \sqrt{\frac{h}{R}}\exp{\Big[\frac{\mu R}{h}\log(Rh)+\frac{\mu R}{h}\log\mu\Big]}
\end{equation}
since the exponential absorbs the terms multiplying by $\sqrt{\mu}$ in front of the exponential. Now we have
 \begin{align*}
 &h^d \int_0^1  \sum_{\substack{j\in \Z^d \\ R-2<|jh|<R+1}} |u_j(t)|^2\,  dt = h^d\int_0^1 \sum_{\substack{j\in \Z^d \\ R-2<|jh|<R+1}} \prod_{k=1}^d\Big(\frac{K_{0}(\frac{2}{eh^2})}{K_{j_k\mu}(\frac{2}{eh^2})}\frac{K_{j_k\mu}(\frac{2}{eh^2})}{K_{0}(\frac{2}{eh^2})}\Big) |u_j(t)|^2\, dt \\
  & \qquad\le h^d \sup_{\substack{j\in \Z^d \\ R-2<|jh|<R+1}} \prod_{k=1}^d\frac{K_{0}(\frac{2}{eh^2})}{K_{j_k\mu}(\frac{2}{eh^2})}\, \sup_{t\in[0,1]}  \sum_{j \in \mathbb{Z}^d} \prod_{k=1}^d\frac{K_{j_k\mu}(\frac{2}{eh^2})}{K_{0}(\frac{2}{eh^2})} |u_j(t)|^2.
\end{align*}
By \eqref{eq:labuena}, the quantity with $\sup_{t\in[0,1]}$ is bounded uniformly in $h$. On the other hand, considering the tensorial product $\prod_{k=1}^d\frac{K_{0}(\frac{2}{eh^2})}{K_{j_k\mu}(\frac{2}{eh^2})}$ for a fixed $j\in \Z^d$ such that $|jh|\simeq R$ we have, by \eqref{eq:estimateKjm}, \eqref{eq:asyK0}, and \eqref{monotonicityBessel}
\begin{align*}
\frac{K_{0}(\frac{2}{eh^2})}{K_{j_1\mu}(\frac{2}{eh^2})}\prod_{k=2}^d\frac{K_{0}(\frac{2}{eh^2})}{K_{j_k\mu}(\frac{2}{eh^2})}&\lesssim h \sqrt{\frac{R}{h}}\exp[-\frac{c}{h^2}]\exp{\Big[-\frac{\mu R}{h}\log(Rh)-\frac{\mu R}{h}\log\mu\Big]}\\
&\lesssim \exp{\Big[-c_0\Big(\frac{\mu R}{h}\log(Rh)+\frac{\mu R}{h}\log\mu\Big)\Big]}
\end{align*}
which leads to the desired estimate.
\section{Semidiscrete heat equation: lower bound}
\label{sec:lower}

The lower bound in Theorem \ref{Thm:lowerbound:heat} will be obtained through a Carleman estimate, which is stated and proven in the next subsection.

\subsection{A Carleman inequality for the heat operator }

\begin{thm}[\textbf{Carleman inequality}]\label{ThCarlemanHeat-h}
Let $R\ge 1$ and $f:(h\mathbb{Z})^d \times [0,1] \to \R$ be such that
$$
\supp(f) \subset \Big\{ 1 \le \big|\frac{hj}{R} + \varphi(t) e_1 \big|  \le 4 \Big\}  \times (0,1).
$$
Let $\phi: (h\mathbb{Z})^d \times [0,1]\to \mathbb{R}$, $\phi_j(t) = \alpha \big| \frac{hj}{R} +\varphi(t) e_{1} \big|^2,$ where $\varphi : [0,1] \to [0,\infty)$ is a smooth function such that $\supp (\varphi) \subset (0,1)$. Then there exists $h_0>0$ and $C>1$  such that
\begin{align}\label{Carleman:heat}
&\frac{1}{h^2} \sqrt{\sinh \frac{ 2\alpha h^2}{R^2}}  \sinh \Big( \frac{2\alpha h}{R\sqrt{d}} \Big) \|f \|_{L^2([0,1]:\ell^2(h\mathbb{Z}^d))} \\
\nonumber
&\quad +  \frac{2}{h^2} \sqrt{ \sinh  \frac{2\alpha h^2}{R^2} } \Big( h^d \int_0^1 \sum_{j\in \mathbb{Z}^d} \sum_{k=1}^d\Big|\frac{f_{j+e_k} -f_{j-e_k} }{2} \Big|^2 dt  \Big)^{1/2}  \le C \|e^{\phi} ( \partial_t -\Delta_{\dis}) (e^{-\phi}f) \|_{L^2([0,1]:\ell^2(h\mathbb{Z}^d))}
\end{align}
where $\alpha$, $0<h<h_0$, and $R$ satisfy the relations
\begin{equation}
\label{eq: carlalpha}
\alpha \le c_\varphi \frac{1}{h^4} \sinh\Big( \frac{2\alpha h^2}{R^2}\Big) \sinh^2\Big( \frac{2\alpha h}{R\sqrt{d}} \Big)
\end{equation}
for a constant $c_\varphi=c_\varphi(d,\|\varphi'\|_\infty,\|\varphi''\|_\infty)$ and
\begin{equation}
\label{eq:peque}
\alpha \ge c R^2 \qquad  \text{ if } \quad \frac{\alpha  h}{R}\le 1/10
\end{equation}
or
\begin{equation}
\label{eq:grande}
1\lesssim \frac{1}{Rh}e^{(2-\varepsilon)\frac{\alpha h}{R\sqrt{d}}}\qquad  \text{ if }\quad  \frac{\alpha  h}{R}\ge \sqrt{d}/2
\end{equation}
for a small, universal $\varepsilon$ such that $h/R<\varepsilon/\sqrt{d}$.
\end{thm}

\begin{proof}
We argue in two steps. First we derive the commutator contributions and in the second step we explain how to absorb the unsigned contributions.

\emph{Step 1: Commutator.}
Let $\phi_j(t) = \alpha \big| \frac{hj}{R} +\varphi(t) e_{1} \big|^2,$ where $\varphi : [0,1] \to [0,\infty)$ is a smooth function supported in $[\frac{1}{4},\frac{3}{4}]$.
As explained, we will often omit the time variable to simplify notation, that is, we will frequently write $f_j:=f_j(t)$. In view of \cite{BFV17}, we can decompose $
 -e^{\phi}\Delta_d(e^{-\phi}f)=Sf+Af$,
 where
 $$
 S \, f _j  =\frac{1}{h^2}\Big\{ -2d f_j + \sum_{k=1}^d \cosh\big( \frac{2\alpha}{R} \big( \frac{j
_k+1/2}{R} +\varphi(t) \delta_{1k} \big) \big)  f_{j+e_k} + \sum_{k=1}^d \cosh\big( \frac{2\alpha}{R} \big( \frac{j_k-1/2}{R} +\varphi(t) \delta_{1k} \big) \big)  f_{j-e_k}\Big\} $$
and
$$A \, f_j=\frac{1}{h^2} \Big\{- \sum_{k=1}^d \sinh\big( \frac{2\alpha}{R} \big( \frac{j_k+1/2}{R} +\varphi(t) \delta_{1k} \big) \big)  f_{j+e_k} + \sum_{k=1}^d \sinh\big( \frac{2\alpha}{R} \big( \frac{j_k-1/2}{R} +\varphi(t) \delta_{1k} \big) \big)  f_{j-e_k}\Big\}. $$
Thus, we have $
e^{\phi}( \partial_t -\Delta_d) g=-\partial_t\phi f+\partial_tf+Sf+Af=\widetilde{S}f+\widetilde{A}f$,
with $\widetilde{S}=S-\partial_t\phi$, $\widetilde{A}=A+\partial_t$.

We compute the expression $\langle[\widetilde{S},\widetilde{A}]f, f\rangle_{L^2([0,1]:\ell^2(h\mathbb{Z}^d))}= (I)+(II)+(III)+(IV)$,
where the terms $(I),\, (II),\, (III)$ and $(IV)$ are explained and discussed in the sequel. More precisely, we have
\begin{equation*}
(I) := \langle[-\partial_t\phi ,  \partial_t  ]f,f\rangle,\quad
(II):= \langle[-\partial_t\phi ,   A ] f, f\rangle,\quad
(III):=\langle[S, \partial_t] f, f\rangle,\quad
(IV):= \langle[S,A]f,f \rangle .
\end{equation*}
We next study these contributions individually. By the observations from \cite{BFV17}, it is known that $(IV)$ gives rise to positive contributions, it will thus be our main aim to either deduce positivity also for $(I)-(III)$ or to absorb the possibly non-negative contributions into $(IV)$.

We begin by computing $(I)$
\begin{align*}
(I)= h^d \int_0^1 \sum_{j\in \mathbb{Z}^d}  \partial_t^2 \phi_j(t) |f_j |^2 dt = 2\alpha h^d  \int_0^1 \sum_{j\in \mathbb{Z}^d} \Big( (\varphi'(t))^2 + \Big(\frac{hj_1}{R}+\varphi(t) \Big) \varphi''(t)  \Big)  |f_j|^2 dt .
\end{align*}
Next, for $(II)$ we obtain
\begin{align*}
(II)&=   2 \alpha h^{d-2}  \int_0^1 \sum_{j\in \mathbb{Z}^d}  \sum_{k=1}^d \sinh\Big( \frac{2\alpha h}{R} \Big( \frac{h(j_k+1/2)}{R} +\varphi(t) \delta_{1k}\Big) \Big)  \Big( \frac{hj_1}{R} +\varphi(t) \Big) \varphi'(t) \, f_{j+e_k}  f_j \,dt \\
& \quad -2 \alpha h^{d-2}  \int_0^1 \sum_{j\in \mathbb{Z}^d}  \sum_{k=1}^d \sinh\Big( \frac{2\alpha h}{R} \Big( \frac{h(j_k-1/2)}{R} +\varphi(t) \delta_{1k} \Big) \Big)   \Big( \frac{hj_1}{R} +\varphi(t) \Big) \varphi'(t) \,f_{j-e_k}f_j \,dt\\
&=- \frac{2 \alpha h}{R}  h^{d-2}  \int_0^1 \sum_{j\in \mathbb{Z}^d}   \sinh\Big( \frac{2\alpha h}{R} \Big( \frac{h(j_1+1/2)}{R} +\varphi(t)  \Big) \Big) \varphi'(t) f_jf_{j+e_1} \,dt.
\end{align*}
For $(III)$ we infer
\begin{align*}
(III)&=2\langle S f, \partial_t f\rangle = -4d \, h^{d-2}  \int_0^1 \sum_{j\in \mathbb{Z}^d}  \, f_j\,\partial_t f_j\, dt\\
& \quad +2 h^{d-2}  \int_0^1 \sum_{j\in \mathbb{Z}^d}  \sum_{k=1}^d \cosh\Big( \frac{2\alpha h}{R} \Big( \frac{h(j
_k  +1/2)}{R} +\varphi(t) \delta_{1k} \Big) \Big)  f_{j+e_k}\partial_t f_j\, dt \\
&\quad +2 h^{d-2}  \int_0^1 \sum_{j\in \mathbb{Z}^d}  \sum_{k=1}^d \cosh\Big( \frac{2\alpha h}{R} \Big( \frac{h(j_k-1/2)}{R} +\varphi(t) \delta_{1k}\Big) \Big)  f_{j-e_k} \partial_t f_j \, dt \\
&=: (III_1)+(III_2)+(III_3).
\end{align*}
In $(III)$ we observe the first term $(III_1)=0$ (since after integration by parts in time we get $(III_1)=-(III_1)$).
For the second sum in $(III)$ we also integrate by parts in time and obtain that
\begin{align*}
(III_2) &= -  2 h^{d-2}  \int_0^1 \sum_{j\in \mathbb{Z}^d}  \sum_{k=1}^d  \Big[ \partial_t \cosh\Big( \frac{2\alpha h}{R}\Big( \frac{h(j
_k+1/2)}{R} +\varphi(t) \delta_{1k}\Big) \Big) \Big] \,  f_{j+e_k}    f_j\, dt  \\
&\quad -2 h^{d-2}  \int_0^1 \sum_{j\in \mathbb{Z}^d}  \sum_{k=1}^d \cosh\Big( \frac{2\alpha h}{R} \Big( \frac{h(j
_k+1/2)}{R} +\varphi(t) \delta_{1k} \Big)\Big) (\partial_t f_{j+e_k}) \,  f_j\, dt  \\
&= -  2 h^{d-2}  \int_0^1 \sum_{j\in \mathbb{Z}^d}  \sum_{k=1}^d   \frac{2\alpha h}{R} \varphi'(t) \delta_{1k} \sinh\Big( \frac{2\alpha h}{R}\Big( \frac{h(j
_k+1/2)}{R} +\varphi(t) \delta_{1k}\Big) \Big)    f_{j+e_k}    f_j\, dt  \\
&\quad-2 h^{d-2}  \int_0^1 \sum_{j\in \mathbb{Z}^d}  \sum_{k=1}^d \cosh\Big( \frac{2\alpha h}{R} \Big( \frac{h(j
_k-1/2)}{R} +\varphi(t) \delta_{1k} \Big)\Big) (\partial_t f_j) f_{j-e_k}\, dt.
\end{align*}
Now, observe that the last term above is just $-(III_3)$ and thus
$$
(III)=(III_2) + (III_3)=-  \frac{4\alpha h}{R} h^{d-2}  \int_0^1 \sum_{j\in \mathbb{Z}^d}     \sinh\Big( \frac{2\alpha h}{R} \Big( \frac{h(j
_1+1/2)}{R} +\varphi(t)  \Big) \Big)    \varphi'(t)  f_{j+e_1}    f_j\, dt,
$$
so that $(III)=2(II)$.

Finally, for $(IV)$, we use observations from \cite{BFV17} adapted to the rescaled setting. This yields
 \begin{align*}
(IV)&=  4h^{d-4} \sinh \frac{ 2\alpha h^2 }{R^2}  \int_0^1 \sum_{j\in \mathbb{Z}^d} \sum_{k=1}^d \sinh^2 \Big( \frac{2\alpha h}{R} \Big(  \frac{hj_k}{R}+\varphi(t) \delta_{1k} \Big)\Big) |f_j|^2dt \\
  &\quad +  4 h^{d-4} \sinh  \frac{2\alpha h^2}{R^2}  \int_0^1 \sum_{j\in \mathbb{Z}^d} \sum_{k=1}^d\Big|\frac{f_{j+e_k} -f_{j-e_k} }{2} \Big|^2 dt.
\end{align*}
Combining the computations for $(I)-(IV)$, we obtain, for the full parabolic commutator,
\begin{align*}
& \langle [\widetilde{S},\widetilde{A}]f, f \rangle = 2\alpha h^d  \int_0^1 \sum_{j\in \mathbb{Z}^d} \Big( (\varphi'(t))^2 + \Big(\frac{hj_1}{R}+\varphi(t) \Big) \varphi''(t) \Big)  (f_j  )^2 dt\\
& \quad -  6 \alpha h^{d-2}  \int_0^1 \sum_{j\in \mathbb{Z}^d}  \sinh\Big( \frac{2\alpha h}{R} \Big( \frac{h(j_1+1/2)}{R} +\varphi(t)  \Big)\Big)  \frac{h}{R} \varphi'(t) \, f_{j+e_1}  f_j dt \\
&\quad +4 h^{d-4}  \sinh \frac{ 2\alpha h^2}{R^2}  \int_0^1 \sum_{j\in \mathbb{Z}^d} \sum_{k=1}^d \sinh^2\Big( \frac{2\alpha h}{R}\Big(  \frac{hj_k}{R}+\varphi(t) \delta_{1k} \Big)\Big) |f_j|^2dt \\
  &\quad + 4h^{d-4} \sinh  \frac{2\alpha h^2 }{R^2}  \int_0^1 \sum_{j\in \mathbb{Z}^d} \sum_{k=1}^d\Big|\frac{f_{j+e_k} -f_{j-e_k} }{2} \Big|^2 dt.
\end{align*}

As already indicated above, we view $(IV)$ as the dominant term and seek to control all possibly non-signed commutator contributions by this term. Thus, we search for sufficient conditions on the parameters such that all not correctly signed contributions from $(I)-(III)$ can be absorbed into $\frac{1}{2}(IV)$. Since all terms carry a factor $h^d$, in the sequel, we will simply drop this in comparing the different contributions. As above, we will deduce sufficient conditions allowing to absorb the possibly non-signed terms by studying the non-signed contributions individually next and compare their coefficients to the ones from $(IV)$.

\emph{Step 2: Absorption arguments.}
For the first contribution $(I)$ we only seek to absorb the second term, since the first one is positive. Thus, comparing the coefficients of the second term in $(I)$ with the first term in $(IV)$, we are led to impose the following sufficient condition
\begin{multline*}
 2\alpha  \| \varphi''\|_{\infty}\int_0^1\sum_{j\in\Z^d}\Big|\frac{hj_1 }{R}+\varphi(t) \Big||f_j|^2  \\\le \frac{2}{3} \frac{1}{h^4}  \sinh \frac{ 2\alpha  h^2}{R^2} \int_0^1\sum_{j\in\Z^d}\sum_{k=1}^d    \sinh^2 \Big( \frac{2\alpha h}{R} \Big(  \frac{hj_k}{R}+\varphi(t) \delta_{1k} \Big)\Big)|f_j|^2 .
\end{multline*}

Due to the support of $f$, it turns out that for any point in the support we have $\big|\frac{hj_k }{R}+\varphi(t)\delta_{1k} \big|\le 4$, $\forall k\in\{1,\dots,d\}$, while there exists $l\in\{1,\dots,d\}$ such that $\big|\frac{hj_l }{R}+\varphi(t)\delta_{1l} \big|\ge \frac{1}{\sqrt{d}}$. This implies that the first contribution can be absorbed as long as
\begin{equation}
\label{eq:condI}
8\alpha \|\varphi''\|_\infty \le \frac{2}{3h^4}\sinh\Big( \frac{2\alpha h^2}{R^2}\Big) \sinh^2\Big( \frac{2\alpha h}{R\sqrt{d}} \Big).
\end{equation}

For $(II)$ we  first apply Young's inequality $2f_{j+e_1}  f_j \le (f_{j+e_1} )^2 +(f_j)^2$ and, after translating in $j$, we need the condition
\begin{multline}
\label{eq:condII}\frac{\alpha}{h R}\|\varphi'\|_\infty \int_0^1\sum_{j\in\Z^d}\Big| \sinh\Big( \frac{2\alpha h}{R}\Big( \frac{h(j_1\pm 1/2)}{R} +\varphi(t) \Big) \Big)\Big| |f_j|^2\\
\le   \frac{1}{3h^4}\sinh \frac{ 2\alpha h^2}{R^2} \int_0^1\sum_{j\in\Z^d}\sum_{k=1}^d \sinh^2 \Big( \frac{2\alpha h}{R} \Big(  \frac{hj_k}{R}+\varphi (t)\delta_{1k} \Big)\Big)|f_j|^2.
\end{multline}
We are reduced to analyze under which assumptions on $\alpha, h$ and $R$, inequalities \eqref{eq:condI} and \eqref{eq:condII} hold. Let us consider two cases. Without loss of generality, let $\|f_j\|_{L^2([0,1]:\ell^2(h\mathbb{Z}^d))}=1$ below.

\noindent
\textit{Case 1}: $\frac{\alpha  h}{R}\le \frac{1}{10}$. Thanks to the support condition, the left hand side of \eqref{eq:condII} is bounded from above by \begin{equation}
\label{eq:1}
\|\varphi'\|_\infty h^{-2}\frac{\alpha  h}{R}\sinh\Big(\frac{9\alpha  h}{R}\Big)
\end{equation}
and the right hand side of \eqref{eq:condII} is bounded from below by
\begin{equation}
\label{eq:2}
\frac23h^{-4}\sinh\Big(\frac{2\alpha  h^2}{R^2}\Big)\sinh^2\Big(\frac{2\alpha  h}{R\sqrt{d}}\Big).
\end{equation}
Since $\sinh(x)\sim x$ for $x$ small, we have that \eqref{eq:1} is bounded by \eqref{eq:2} whenever $\alpha \ge c R^2$ (in particular when $\alpha =CR^2$ for a suitable constant $C=C(\|\varphi'\|_\infty)>0$).

\noindent
\textit{Case 2}: $\frac{\alpha  h}{R}\ge \frac{\sqrt{d}}{2}$. Since $h/R$ is small, we can assume that there exists a small, universal $\varepsilon$ such that $h/R<\varepsilon/\sqrt{d}$\footnote{Indeed, we can assume that $\varepsilon/\sqrt{d}=1/M$, where $M$ is the constant in Theorem \ref{thm:UpperBound}.}.  For each $j=(j_1, \ldots, j_d)$, we distinguish two situations.

\noindent
\textit{Situation 2.1: the component $j_1$ satisfies $\big|\frac{hj_1 }{R}+\varphi(t) \big|\ge \frac{1}{\sqrt{d}}$}. Let us look at the argument of $\sinh$ in the left hand side of \eqref{eq:condII}, namely $
\frac{2\alpha h}{R}\Big( \frac{h(j_1\pm 1/2)}{R} +\varphi(t)  \Big)=\frac{2\alpha h}{R}\Big( \frac{hj_1}{R} +\varphi(t) \Big)\pm\frac{\alpha h^2}{R^2}$.
Hence,
$$
\frac{\alpha h^2}{R^2}\le \frac{\alpha h}{R}\frac{\varepsilon}{\sqrt{d}}\lesssim \varepsilon \frac{\alpha h}{R}\Big|\frac{hj_1 }{R}+\varphi(t)\Big|
$$
and so the term on left hand side of \eqref{eq:condII} is bounded from above by  a constant times
$$
h^{-2}\frac{\alpha  h}{R}\sinh\Big((2+\varepsilon)\frac{\alpha h}{R}\Big|\frac{hj_1 }{R}+\varphi(t)\Big|\Big).
$$
Now we would like to use that $\sinh(x)\sim e^{x}/2$ for $x$ large (say $x\ge1$), this being possible since
$$
(2+\varepsilon)\frac{\alpha h}{R}\Big|\frac{hj_1 }{R}+\varphi(t)\Big|> \frac{\alpha h}{R}\frac{2}{\sqrt{d}}.
$$
Thus
\begin{equation}
\label{eq:3}
h^{-2}\frac{\alpha  h}{R}\sinh\Big((2+\varepsilon)\frac{\alpha h}{R}\Big|\frac{hj_1 }{R}+\varphi(t)\Big|\Big)\le h^{-2}\frac{\alpha  h}{R}e^{(2+\varepsilon)\frac{\alpha h}{R}\big|\frac{hj_1 }{R}+\varphi(t)\big|}.
\end{equation}
 On the other hand, the term corresponding with the component $j_1$ on the right hand side of \eqref{eq:condII} is bounded from below by a constant times
\begin{multline}
\label{eq:4}
h^{-4}\sinh\Big(\frac{2\alpha  h^2}{R^2}\Big)\sinh^{1+\varepsilon/2}\Big( \frac{2\alpha h}{R} \Big|  \frac{hj_1}{R}+\varphi(t) \Big|\Big)\sinh^{1-\varepsilon/2}\Big( \frac{2\alpha h}{R} \Big|  \frac{hj_1}{R}+\varphi(t) \Big|\Big)\\
\ge h^{-4}\sinh\Big(\frac{2\alpha  h^2}{R^2}\Big)e^{(2+\varepsilon) \frac{\alpha h}{R} \big|  \frac{hj_1}{R}+\varphi(t) \big|}\sinh^{1-\varepsilon/2}\Big( \frac{2\alpha h}{R}\frac{1}{\sqrt{d}}\Big).
\end{multline}
Thus \eqref{eq:3} is bounded by \eqref{eq:4} whenever
\begin{equation}
\label{eq:critical}
h^{-2}\frac{\alpha  h}{R}\lesssim h^{-4}\sinh\Big(\frac{2\alpha  h^2}{R^2}\Big)\sinh^{1-\varepsilon/2}\Big( \frac{2\alpha h}{R}\frac{1}{\sqrt{d}}\Big).
\end{equation}
Nevertheless, since $x\le \sinh x$ for all positive $x$, it can be easily checked that the condition $
1\lesssim \frac{1}{Rh}e^{(2-\varepsilon)\frac{\alpha h}{R\sqrt{d}}}$
implies \eqref{eq:critical}.

\noindent
\textit{Situation 2.2: the component $j_1$ satisfies $\big|\frac{hj_1}{R}+\varphi(t) \big|\le \frac{1}{\sqrt{d}}$}. Recall that there exists $l\neq 1$ such that $\big|\frac{hj_l }{R}+\varphi(t)\delta_{1l} \big|\ge \frac{1}{\sqrt{d}}$. Then the left hand side of \eqref{eq:condII} is bounded from above by a constant times $
h^{-2}\frac{\alpha  h}{R}\sinh\big(\frac{2\alpha h}{R}\frac{1}{\sqrt{d}}+\frac{\alpha h}{R}\frac{\varepsilon}{\sqrt{d}}\big)\le h^{-2}\frac{\alpha  h}{R}\sinh\big((2+\varepsilon)\frac{\alpha  h}{R\sqrt{d}}\big)$.
On the other hand, the right hand side of \eqref{eq:condII} is bounded from below by a constant times $
h^{-4}\sinh\big(\frac{2\alpha  h^2}{R^2}\big)\sinh^2\big(\frac{2\alpha  h}{R\sqrt{d}}\big)$.
Therefore, we require that
\begin{equation}
\label{eq:critical2}
h^{-2}\frac{\alpha  h}{R}e^{(2+\varepsilon)\frac{\alpha  h}{R\sqrt{d}}}\lesssim h^{-4}\frac{2\alpha  h^2}{R^2}e^{\frac{4\alpha  h}{R\sqrt{d}}}
\end{equation}
and, as before, the condition $
1\lesssim \frac{1}{Rh}e^{(2-\varepsilon)\frac{\alpha h}{R\sqrt{d}}}$
is sufficient to ensure \eqref{eq:critical2}. The study of $(III)$ is identical to that of $(II)$ and we arrive at the same condition, up to a dimensional constant.

If these conditions (which are the conditions given at \eqref{eq: carlalpha})  are satisfied, then the unsigned contributions of the commutator are absorbed in a constant multiple of $(IV)$, which is bounded by  the left-hand side of \eqref{Carleman:heat}. This finishes the proof since
$$
 \|e^{\phi} \big( \partial_t -\Delta_d\big) g \|^2_{L^2([0,1]:\ell^2(h\mathbb{Z}^d))}= \|\widetilde{S}f+\widetilde{A}f \|^2_{L^2([0,1]:\ell^2(h\mathbb{Z}^d))}\ge \langle[\widetilde{S},\widetilde{A}]f, f\rangle_{L^2([0,1]:\ell^2(h\mathbb{Z}^d))},
$$
as desired.
\end{proof}

\subsection{Lower bound: proof of Theorem \ref{Thm:lowerbound:heat}}
\label{lower}

The quantitative lower bound in Theorem \ref{Thm:lowerbound:heat} will be derived from the Carleman estimate in Theorem \ref{ThCarlemanHeat-h}.
We first state and prove a technical lemma.

\begin{lem}
\label{lem:crucial}
Let $u\in C^1([0,1]:\ell^2(h\mathbb{Z}^d))$ be a non-trivial solution to \eqref{p1}.
Assume that $R \geq 1$. There exists $h_0>0$ such that if \eqref{eq: carlalpha},\eqref{eq:peque}, \eqref{eq:grande} holds for $h\in(0,h_0)$, and for some $C>1$,
\begin{equation}\label{hide:VC}
C\big(u(0)+\|V\|_{\infty}^2\big) \le \frac{1}{h^4} \sinh \frac{ 2\alpha h^2}{R^2}   \sinh^2 \Big( \frac{2\alpha h}{R\sqrt{d}} \Big),
\end{equation}
 are satisfied, then
\begin{equation}\label{lowerboundlemma}
 \frac{1}{h^4} \sinh \frac{ 2\alpha h^2}{R^2}   \sinh^2 \Big( \frac{2\alpha h}{R\sqrt{d}}\Big)   e^{-  14\alpha}  \le Ch^d\int_0^1\sum_{\substack{j\in \Z^d \\ R-2<|hj|<R+1}} (|u_j(0)|^2+|u_j(t)|^2) \,dt.
\end{equation}
\end{lem}

\begin{proof}
Let $\theta_R:\R^d\to \R$ be a smooth function such that $\theta_R\in C_0^{\infty}(\R^d)$, $\theta_R(x)=0$ for $|x|\ge R$ and $\theta_R(x)=1$ for $|x|\le R-1$.  Let $\phi:\R^d \times [0,\infty)\to \R$ be defined
by
$$
\phi(x,t)=\alpha \Big|\frac{x}{R}+\varphi(t)e_1\Big|^2,
$$
where  $\varphi : [0,1] \to [0,\infty)$ is a smooth function such that
$$
0 \le \varphi \le 3, \quad \varphi(t) = 0 \text{ for } t\in \Big[0, \frac{1}{4}\Big]\cup\Big[\frac{3}{4},1\Big] , \quad \varphi(t) = 3 \text{ for } t\in \Big[ \frac38, \frac58 \Big].
$$
Let $\eta:\R^d \to \R$ be a smooth function such that $\eta(x)=0$ if $|x|\le 1$ and $\eta(x)=1$ if $|x|\ge 2$.
Thus
 $$g_j(t):= u_j(t) \, \theta_R(hj)\, \eta\Big(\frac{hj}{R}+\varphi( t)e_1 \Big)$$
 is a compactly supported function for which the Carleman inequality \eqref{Carleman:heat} holds true.

We will use the decomposition
$$
 \Delta_{\dis,k} (fg)_j =g_j \Delta_{\dis,k} f_j + 2D^s_kf_j D_{+,k} g_j+ f_{j-e_k}\Delta_{\dis,k} g_j.
$$
This will give
\begin{align*}
&(\partial_t- \Delta_{\dis}) [g_j(t) ] =  ( \partial_t-\Delta_{\dis}) [u_j(t)] \, \theta_R(hj)\, \eta\Big(\frac{hj}{R}+\varphi( t)e_1 \Big)    \\
&\quad - 2\sum_{k=1}^d D^s_ku_j(t)D_{+,k} \Big(\theta_R(hj) \eta\Big(\frac{hj}{R}+\varphi( t)e_1  \Big) \Big) - \sum_{k=1}^d u_{j-e_k}(t)\Delta_{\dis,k}  \Big( \theta_R(hj)\, \eta\Big(\frac{hj}{R}+\varphi( t)e_1  \Big)\Big)\\
&\quad + u_j(t) \partial_t \Big(\theta_R(hj) \eta\Big(\frac{hj}{R}+\varphi( t)e_1  \Big) \Big).
\end{align*}

Thus, the Carleman  inequality \eqref{Carleman:heat} gives (notice that the $h^d$ term can be simplified from all terms, so we refrain from writing it and will include it after these computations)
\begin{align}
\label{chorizo}
\notag\frac{1}{Ch^4}& \sinh \frac{ 2\alpha h^2}{R^2}   \sinh^2 \Big( \frac{2\alpha h}{R\sqrt{d}} \Big)  \int_0^1 \sum_{j\in \Zd} ( e^{\phi_j(t)}g_j)^2 dt \\
& \le \int_0^1\sum_{j\in \Zd} \Big(e^{\phi_j(t)} V_j(t) u_j(t)   \, \theta_R(hj)\, \eta\Big(\frac{hj}{R}+\varphi( t)e_1 \Big) \Big)^2    \\
\notag&\quad  + 2 \int_0^1\sum_{j\in \Zd}\sum_{k=1}^d e^{2\phi_j(t) }  (D^s_ku_j(t) )^2 \Big( D_{+,k} \Big(\theta_R(hj) \eta\Big(\frac{hj}{R}+\varphi(t)e_1  \Big) \Big) \Big)^2\\
\notag& \quad + \int_0^1\sum_{j\in \Zd} \sum_{k=1}^de^{2\phi_j(t) } (u_{j-e_k}(t) )^2  \Big( \Delta_{\dis,k} \Big(\theta_R(hj) \eta\Big(\frac{hj}{R}+\varphi(t)e_1  \Big) \Big) \Big)^2 \\
\notag&\quad +  \int_0^1\sum_{j\in \Zd} e^{2\phi_j(t) } (u_j(t) )^2 \Big( \partial_t \Big(\theta_R(hj) \eta\Big(\frac{hj}{R}+\varphi(t)e_1  \Big) \Big)\Big)^2.
\end{align}
First we aim to absorb the term with $V$ into the left hand side. Thus, we need to ensure
\begin{equation}\label{hide:V}
C \|V\|_{\infty}^2 \le \frac{1}{10}\frac{1}{h^4} \sinh \frac{ 2\alpha h^2}{R^2}   \sinh^2 \Big( \frac{2\alpha h}{R\sqrt{d}} \Big) .
 \end{equation}
We analyze each of the other terms individually. We have
\begin{align*}
D_{+,k} \Big(\theta_R(hj) \eta\Big(\frac{hj}{R}+\varphi(t)e_1 \Big) \Big)  =D_{+,k} \theta_R(hj) \, \eta\Big(\frac{hj}{R}+\varphi(t)e_1 \Big) +\theta_R((j+e_k)h) D_{+,k} \eta\Big(\frac{hj}{R}+\varphi(t)e_1  \Big).
\end{align*}
Notice that $D_{+,k} \theta_R$ is supported inside $\{ R-1{-h}< |hj|<R{+h}\}$ and $
| D_{+,k} \theta_R(hj)| \le  \| \nabla \theta_R \| _{\infty}  \le C$.
Thus in this region $|\frac{hj}{R}+\varphi e_1|\le 4+h$ and, by inequality \eqref{Caccioppoli:rings}, we obtain, for $h_0$ small enough,
 \begin{multline*}
   \int_0^1\sum_{j\in \Zd} e^{2\phi_j(t)}  ( D^s_ku_j(t))^2 \Big( D_{+,k} \theta_R(hj) \, \eta\Big(\frac{hj}{R}+\varphi(t)e_1 \Big) \Big)^2 \\
    \le C  e^{33 \alpha} \int_0^1\sum_{R-2<|hj|<R+1} (|u_j(0)|^2+|u_j(t)|^2) \,dt.
   \end{multline*}
We observe that the support of $D_{+,k}\eta\big(\frac{hj}{R}+\varphi(t)e_1 \big)$ is contained in $\{1-\frac{h}{R} \le |\frac{hj}{R}+\varphi e_1| \le 2+\frac{h}{R}\} $
and $
\big | D_{+,k}\eta\big(\big|\frac{hj}{R}+\varphi(t)e_1\big | \big)\big| \le C \frac{1}{R}$.
Hence, for any $k=1,\dots,d$,
\begin{align*}
&\int_0^1\sum_{j\in \Zd} e^{2\phi_j(t)} (D^s_k u_j(t))^2 \Big( \theta_R((j+e_k)h)D_{+,k}\eta\Big(\frac{hj}{R}+\varphi(t)e_1 \Big)\Big)^2 dt \\
&\quad \le C   \frac{e^{2\alpha(2+\frac{h}{R})^2}}{R^2}  \int_0^1\sum_{\{|hj|<R+h\} \cap \{ 1-\frac{h}{R} \le |\frac{hj}{R}+\varphi e_1| \le 2+\frac{h}{R} \}  } ( D^s_k u_j(t))^2 \,dt  \le C    \frac{e^{2\alpha(2+\frac{h}{R})^2}}{R^2}\sum_{j \in \mathbb{Z}^d }  |u_j(0)|^2,
\end{align*}
where we used the energy estimate (Lemma \ref{lem:energy}) in the last inequality and $C$ depends on $\|V\|_{\infty}$.

For next terms we proceed similarly, delivering that for any $k=1,\dots,d$,
\begin{multline*}
\int_0^1\sum_{j\in \Zd} e^{2\phi_j(t)}  |u_{j-e_k}(t)|^2 \Big(\Delta_{\dis,k} \Big( \theta_R(hj) \,\eta\Big(\frac{hj}{R}+\varphi(t)e_1\Big)\Big)\Big)^2  \\
\quad \le C e^{33 \alpha} \int_0^1\sum_{ R-1{-h}< |hj|<R{+h}} |u_{j-e_k}(t)|^2 \,dt+ C \frac{e^{2\alpha(2+\frac{h}{R})^2}}{R^2} \int_0^1\sum_{\{ 1-\frac{h}{R} \le |\frac{hj}{R}+\varphi e_1| \le 2+\frac{h}{R} \} } |u_{j-e_k}(t)|^2 \,dt.
\end{multline*}
Finally, for the last term notice that $\partial_t\big[ \eta\big(\frac{hj}{R}+\varphi(t)e_1  \big)   \big] =   \partial_{x_1}\eta\big( \frac{hj}{R}+\varphi(t)e_1  \big) \varphi'(t) $
so that
\begin{multline*}
\int_0^1\sum_{j\in \Zd}  e^{2\phi_j(t)}   (u_j(t)\, \theta_R(hj))^2 \Big(  \partial_{x_1}\eta\Big( \frac{hj}{R}+\varphi(t)e_1 \Big) \varphi'(t) \Big)^2 \\
\le C  \|\varphi'\|_{\infty} \|\nabla\eta\|_{\infty}  \, e^{8 \alpha} \int_0^1\sum_{|jh|<R } |u_j(t)|^2 \,dt.
\end{multline*}
For the left-hand of  side \eqref{chorizo} we notice that $g\equiv u$ in the set $\{ |hj|\le R-1 \}\times [\frac{1}{2}-\frac{1}{8}, \frac{1}{2}+\frac{1}{8}].$ In this set $\varphi\equiv 3$ and thus  $|\frac{hj}{R}+\varphi(t)e_1 |  \ge 3-\frac{R-1}{R} = 2+\frac{1}{R} .$    Thus,
$$
 \int_0^1 \sum_{j\in \Zd} ( e^{\phi_j(t)}g_j)^2 dt \ge  \int_{\frac{1}{2}-\frac{1}{8}}^{ \frac{1}{2}+\frac{1}{8}} \sum_{|hj|<R-1} (e^{\phi_j(t)}u_j(t))^2 \, dt  \ge e^{2 (2+\frac{1}{R})^2 \alpha} \int_{\frac{1}{2}-\frac{1}{8}}^{ \frac{1}{2}+\frac{1}{8}} \sum_{|hj|<R-1} |u_j(t)|^2 \, dt.
 $$
Since $u$ is nontrivial, after a suitable scaling one can assume that
$$
h^d\int_{\frac{1}{2}-\frac{1}{8}}^{ \frac{1}{2}+\frac{1}{8}} \sum_{|hj|<R-1}|u_j(t)|^2 \, dt \ge 1.
$$
Indeed, if the latter integral is zero for every $R$, automatically $u\equiv0$. So there must exist some $R_0$ and some positive constant $c$ such that for $R\ge R_0$ the integral is greater than $c$. Since a constant times the solution solves the same equation, one can assume without loss of generality that $c\ge 1$.

Altogether we obtain, from \eqref{chorizo}, after including the factor $h^d$ that we removed,
\begin{multline*}
 \frac{1}{h^4} \sinh \frac{ 2\alpha h^2}{R^2}   \sinh^2\Big( \frac{2\alpha h}{R\sqrt{d}}\Big) e^{2 (2+\frac{1}{R})^2 \alpha} \\
 \le  Ch^d\Big(\frac{e^{2\alpha(2+\frac{h}{R})^2}}{R^2} \sum_{j\in\Zd }  |u_j(0)|^2+  e^{33 \alpha} \int_0^1\sum_{R-2<|hj|<R+1} (|u_j(0)|^2+|u_j(t)|^2) \,dt\Big),
\end{multline*}
where we also used the energy estimate in Lemma \ref{lem:energy}, namely $\int_0^1\sum_{j\in\Zd }  |u_j(t)|^2\,dt \le C\sum_{j\in\Zd }  |u_j(0)|^2 $.
Thus, if
\begin{equation}\label{hide:C}
\frac12 \frac{1}{h^4}\sinh \frac{ 2\alpha h^2}{R^2}   \sinh^2 \Big( \frac{2\alpha h}{R\sqrt{d}} \Big) \ge  \frac{\|u(\cdot,0)\|_{\ell^2}^2}{R^2}
\end{equation}
then for $R\ge1$
$$
\frac12\frac{1}{h^4} \sinh \frac{ 2\alpha h^2}{R^2}   \sinh^2\Big( \frac{2\alpha h}{R\sqrt{d}}\Big)   e^{-  14\alpha}  \le h^d\int_0^1\sum_{R-2<|hj|<R+1} (|u_j(0)|^2+|u_j(t)|^2) \,dt.
$$
Now, to summarize, we just need to ensure that conditions \eqref{hide:V} and \eqref{hide:C} hold, or stated altogether, that
$$
C\big(u(0),\|V\|_{\infty}\big) \le \frac{1}{h^4} \sinh \frac{ 2\alpha h^2}{R^2}   \sinh^2 \Big( \frac{2\alpha h}{R\sqrt{d}} \Big)
$$
holds.
\end{proof}

Now we are ready to prove the lower bound of Theorem \ref{Thm:lowerbound:heat}.

\begin{proof}[Proof of Theorem  \ref{Thm:lowerbound:heat}]
We have to optimize the lower bound from the previous Lemma \ref{lem:crucial} in $\alpha$ and $R$. This amounts to studying the relation between $\alpha, h, R$ such that  \eqref{eq: carlalpha} and either \eqref{eq:peque} or \eqref{eq:grande}, hold. In words, besides \eqref{hide:VC} we need the condition  \eqref{eq: carlalpha}
$$
c_\varphi h^{-4}\sinh^2(2\alpha h /R\sqrt{d})  \sinh(2\alpha h^2/R^2) \ge \alpha
$$
and
 $$
\alpha \ge c R^2 \,\,\,  \text{ if } \,\,\, \frac{\alpha  h}{R}\le1/10\qquad  \text{ or } \qquad
1\lesssim \frac{1}{Rh}e^{(2-\varepsilon)\frac{\alpha h}{R\sqrt{d}}}\,\,\,  \text{ if }\,\,\,  \frac{\alpha  h}{R}\ge\sqrt{d}/2
$$
for a small, universal $\varepsilon$ such that $h/R<\varepsilon/\sqrt{d}$.
We have to take into account that $h$ is small, $R$ large, and we would like to have $\alpha$ large, but going to infinity in the slowest possible way. Let us first link $R$ and $h$ by taking $R=h^{-\beta}$ for $\beta>0$ and do some casuistic to determine $\alpha$ at every scale $|hj|\sim R$. Observe that at each scale,  \eqref{eq: carlalpha} reads
\begin{equation}\label{eq:carlalphascale}
c_\varphi R^{4/\beta}\sinh^2(2\alpha /R^{1+1/\beta}\sqrt{d})  \sinh(2\alpha/R^{2+2/\beta}) \ge \alpha.
\end{equation}
We recall that $\sinh(x)\sim x$ and $\sinh(x)\sim e^{|x|}/2$ for $x$ small and large, respectively.\\
\noindent{\textbf{Case 1: Choice of $\alpha$ such that $\alpha\le R^{1+1/\beta}/10$}.} In this case, since $\alpha h/R\le 1/10$, we require $\alpha \ge c R^2$. Besides, all hyperbolic functions are evaluated at small argument. Therefore in this regime we can use the relation $2x\ge \sinh x\ge x$, for $x$ positive, and condition  \eqref{eq:carlalphascale} is satisfied if
 $$
c_\varphi R^{4/\beta}\frac{4\alpha^2}{R^{2+2/\beta} d}\frac{2\alpha}{R^{2+2/\beta}}\ge \alpha \Longleftrightarrow 8c_\varphi \alpha^2 \ge dR^4.
 $$
Hence, a choice to get a constant size and answer positively to our question is
$$
\alpha=cR^{2},
$$
 with $c$ large enough independent of $R$. This is the same as in the continuum regime but this choice forces the requirement $cR^2\le R^{1+1/\beta}/10$ or, in other words, $\beta< 1$. If $\beta=1$ we require some control on the constant $c$ (which depends on the dimension and the function $\varphi$).\\
\noindent{\textbf{Case 2: Choice of $\alpha$ such that $\alpha\ge\sqrt{d}R^{1+1/\beta}/2$ and $\alpha\le R^{2+2/\beta}/10$.}} This is the range for $\alpha$ to satisfy \eqref{eq:grande}, so at the scale $R=h^{-\beta}$ we require
\begin{equation}\label{eq: grandescale}
R^{-1+1/\beta}e^{(2-\varepsilon)\frac{\alpha}{R^{1+1/\beta}\sqrt{d}}}\ge C_1.
\end{equation}
We recall that $\varepsilon$ is a universal small number such that $\varepsilon> \sqrt{d}h/R$. Now, when the hyperbolic sine is evaluated at small argument, we can use as before $2x\ge \sinh x\ge x$, while at large arguments one has $e^x/2\ge \sinh x\ge e^x/4$. Therefore, to ensure \eqref{eq:carlalphascale} we require
 $$
c_\varphi R^{4/\beta}\frac{e^{\frac{4\alpha}{R^{1+1/\beta} \sqrt{d}}}}{4}\frac{2\alpha}{R^{2+2/\beta}}\ge \alpha \Longleftrightarrow \frac{c_\varphi e^{\frac{4\alpha}{R^{1+1/\beta} \sqrt{d}}}}{2R^{2-2/\beta}} \ge 1,
$$
which is a weaker relation than \eqref{eq: grandescale}. Hence, we have to analyze  the best relation in \eqref{eq: grandescale}. If $\beta<1$ we take $\alpha=\sqrt{d}R^{1+1/\beta}/2$, if $\beta=1$ we take $\alpha=cR^2$ with $c=\max\{\frac{\sqrt{d}}{2},\frac{\sqrt{d}}{4}\log\left(\frac{2}{c_\varphi}\right)\}$, and if $\beta>1$ we take $\alpha= \tilde{c} R^{1+1/\beta} \log(R^{1-1/\beta})$, for which we need
$$
R^{((2-\epsilon)\tilde{c}/\sqrt{d}-1)(1-1/\beta)}\ge C_1
$$
which is satisfied if $\tilde{c}>\frac{\sqrt{d}}{2-\epsilon}$.\\
\noindent{\textbf{Case 3: Choice of $\alpha$ such that $\alpha\ge R^{2+2/\beta}$.}} In this range we require \eqref{eq: grandescale} and, since both hyperbolic functions are evaluated at large argument,
$$
c_\varphi R^{4/\beta}e^{\frac{4\alpha}{R^{1+1/\beta}\sqrt{d}}+\frac{2\alpha}{R^{2+2\beta}}} \ge 16\alpha.
$$
The best choice in this case is then $\alpha=R^{2+2/\beta}$.

Combining the three cases, if $\beta\le 1$, $\alpha=\min\{ cR^2, C R^{1+1/\beta},R^{2+2/\beta}\}=c R^2$. In other words, if $Rh\le 1$, the choice of $\alpha$ is the same as in the continuum case. If $\beta>1$, then
$$
\alpha=\min\{\tilde{c} R^{1+1/\beta}\log(R^{1-1/\beta}),R^{2+2/\beta}\}=\tilde{c} R^{1+1/\beta}\log (R^{1-1/\beta})
$$
 with $\tilde{c}$ a universal constant independent of $\beta$. Observe that at these scales $R=h^{-\beta}$, we can take $Rh\ge 1$ and $\alpha=\tilde{c} R^{1+1/\beta}\log(R^{1-1/\beta})=\tilde{c}\frac{R}{h}\log(Rh)$. We summarize this study as
\begin{equation}
\label{eq:lower_bound_bd1}
\alpha=
\begin{cases}
c R^{2} \mbox{ if } Rh\le 1,\\
\tilde{c}\frac{R}{h}\log (Rh) \mbox{ if } Rh>1.
\end{cases}
\end{equation}
Inserting these choices into \eqref{lowerboundlemma}, taking \eqref{eq: carlalpha}  into account, we obtain
\begin{equation}
    Ch^d\int_0^1\sum_{\substack{j\in \Z^d \\ R-2<|hj|<R+1}} (|u_j(0)|^2+|u_j(t)|^2) \,dt \ge e^{-  14\alpha} \alpha\ge\begin{cases}
e^{-c R^{2}} \mbox{ if } Rh\le 1,\\
e^{-\tilde{c}\frac{R}{h}\log (Rh)} \mbox{ if } Rh>1,
\end{cases}
\end{equation}
as desired.

If for fixed $h$ we want to find the best choice of $\alpha$, we notice that only \eqref{eq:grande} can occur, and in order to satisfy this condition, we take $\alpha=c\frac{R}{h}\log(Rh)$. For $c$ large enough, only depending on $d$,  \eqref{eq: carlalpha} is sastisfied. This agrees with the previous lower bound.
\end{proof}

\begin{rmk}
\label{others}
In the case $R=h^{-\beta}$, $Rh>1$ we can write the lower bound as
$$
C h^d\int_0^1\sum_{\substack{j\in \Z^d \\ R-2<|hj|<R+1}} (|u_j(0)|^2+|u_j(t)|^2) \,dt \ge e^{-  14
C R^{1+\frac{1}{\beta}}\log (R^{1-1/\beta})} \quad \mbox{ if } Rh\ge 1.
$$
It is possible to adapt the previous argument to
have $e^{-  14 C R^{1+\frac{1}{\beta}}\log (1+R^{1-1/\beta})}$ on the right hand side. Then, the value $\beta=1$ leads us to the close-to-continuum region, and gives us the $R^{2}$ factor.

\end{rmk}

\section{Landis-type results for the semidiscrete heat equation}
\label{sec:mainproof}

\subsection{The close-to-continuum regime: proof of Theorem \ref{thm:qualitative} (1)}

We present a proof of the Landis-type estimate for the discrete heat equation in Theorem~\ref{thm:qualitative} $(1)$. Assume that $u$ is nontrivial. We focus on the regime $Rh$ small, so in this case, from the lower bound in Theorem~\ref{Thm:lowerbound:heat},
$$
h^d\int_0^1\sum_{\substack{j\in \Z^d \\ R-2<|hj|<R+1}}(|u_j(0)|^2+ |u_j(t)|^2) \,dt\gtrsim
e^{-c R^{2}}.
$$
From the upper bound in Theorem \ref{thm:UpperBound},
$$
 h^d \int_0^1 \sum_{\substack{j\in \Z^d \\ R-2<|hj|<R+1}}(|u_j(0)|^2+ |u_j(t)|^2)\,dt
  \le
C_{\gamma,d}e^{-  dR^2/\gamma}
$$
and hence
$$
e^{-c R^{2}}\le C_{\gamma,d}e^{-dR^2/\gamma}.
$$
Thus, for $R$ greater than some specific $\bar{R}(C,\gamma, d)$ we arrive at a contradiction if $\gamma <dc^{-1}$. Notice this simultaneously gives the existence of some $h_0=h_0(\gamma,c)$ such that for $h\in (0,h_0)$ the regime condition $Rh\le \min\{\gamma/2, 1/10\}$ is satisfied.

\subsection{The purely discrete regime: proof of Theorem \ref{thm:qualitative} (2)}

The proof of Theorem \ref{thm:qualitative} (2) follows the idea of the proof of Theorem \ref{thm:qualitative} (1) and makes use of the upper and lower bounds in Theorems \ref{thm:UpperBound2} and \ref{Thm:lowerbound:heat}, respectively. This time, if we assume $\mu$ to be large enough, only depending on the dimension, by letting $Rh\ggg1$, we reach a contradiction, hence $u\equiv 0$, see also \cite[p. 4864]{BFV17}.

\section{Optimality}
\label{sec:optimal}

We can see with an example how the quantitative estimates we have obtained in Theorems \ref{thm:UpperBound}, \ref{thm:UpperBound2}, and~\ref{Thm:lowerbound:heat} are optimal. For simplicity, we will restrict ourselves to the one dimensional case in this section.
Hence, we are going to provide a nontrivial solution $u_j(t)$ to the discrete free heat equation with fast decay and study the quantity
$$
h\int_0^1\sum_{\substack{j\in \Z \\ R-2<|hj|<R+1}}(|u_j(0)|^2+ |u_j(t)|^2) \,dt.
$$
This quantity behaves as the lower and upper bounds of the results above or, in other words, these bounds are tight. This example does not contradict Theorem \ref{thm:qualitative} $(1)$ because the decay rate $\gamma$ is too large and the hypothesis $\gamma<C$ of this result is not satisfied.

Indeed, let us consider the function
\begin{equation}
\label{eq:ujexample}
 u_j(t)=\sum_{j\in \Z}e^{-2t/h^2}\frac{I_{j}(2t/h^2+1/h^2)}{I_0(1/h^2)}.
\end{equation}
First of all, one can check that this function is a solution to equation \eqref{p1} such that fulfills the condition \eqref{eq:inverseG}. By taking for instance $\gamma=4$ in  \eqref{eq:inverseG}, it suffices to see that the quantities
$$
 h \sum_{j\in \Z}\frac{K_{j}^2(\frac{4}{h^2})I_{j}^2(\frac{1}{h^2})}{K_{0}^2(\frac{4}{h^2})I_{0}^2(\frac{1}{h^2})}\quad \text{ and } \quad  e^{-\frac{4}{h^2}}h \sum_{j\in \Z} \frac{K_{j}^2(\frac{4}{h^2})I_{j}^2(\frac{3}{h^2})}{K_{0}^2(\frac{4}{h^2})I_{0}^2(\frac{3}{h^2})}
$$
are finite, with estimates uniform in $h$. Intuitively, by arguments similar to those appearing in \cite[Theorem 2.1]{FB17} or \cite[Proposition 4.7]{CJK10} (see also Appendix \ref{sub:heats}), we can see that $u_j(t)$ is an approximation of the solution to the free heat equation with Gaussian initial data, and that the function $K_j(4/h^2)/K_0(4/h^2)$ approximates the inverse of another Gaussian function.  Therefore, both conditions above approximate an integral of a certain Gaussian function. In any case, one can check this by means of computations in the spirit of the ones carried out in Subsection \ref{subsec:upperb}. Now the analysis of the sums has to distinguish the regimes in which $j^2h^4\le c_1$ for some $c_1$ positive and small and  $j^2h^4\ge c_2$ for some $c_2$ positive and large, and then make use of the uniform bounds \eqref{eq:Kjlarge} and \eqref{eq:Ijlarge}. For the sum $\sum_{c_1< j^2h^4<c_2}\frac{K_{j}^2(\frac{4}{h^2})I_{j}^2(\frac{1}{h^2})}{K_{0}^2(\frac{4}{h^2})I_{0}^2(\frac{1}{h^2})}$ (respectively, the one concerning the condition on $u_j(1)$), it suffices to notice that the $j\mapsto K_{j}(b)I_{j}(a)$, for $a<b$, is decreasing; this follows from \eqref{eq:monotonicity}.
An analogous computation can be carried out to see that the function $u_j(t)$ in \eqref{eq:ujexample} satisfies that
$$
 h \sum_{j\in \Z}\frac{K_{j\mu}(\frac{2}{eh^2})I_{j}^2(\frac{1}{h^2})}{K_{0}(\frac{2}{eh^2})I_{0}^2(\frac{1}{h^2})}\quad \text{ and }
 \quad  e^{-\frac{4}{h^2}}h \sum_{j\in \Z}\frac{K_{j\mu}(\frac{2}{eh^2})I_{j}^2(\frac{3}{h^2})}{K_{0}(\frac{2}{eh^2}))I_{0}^2(\frac{3}{h^2})}
$$
are finite (and actually uniformly in $h$, for $\mu=1/e$, for instance). This is the condition in Theorem~\ref{thm:qualitative}~$(2)$.

Now, let us consider a point in the lattice such that $jh\simeq R$. Using once again the asymptotic expressions of the Bessel functions one can check that
\[
u_j(t)\sim e^{-\frac{1+2t}{h^2}+\frac{R}{h}\left[\log\left(\frac{1+2t}{hR+\sqrt{h^2R^2+(1+2t)^2}}+\sqrt{1+\frac{(1+2t)^2}{h^2R^2}}\right)\right]}\frac{1}{(h^2R^2+(1+2t)^2)^{1/4}}.
\]
If $Rh$ is small, we proceed as in the first part of Theorem \ref{Thm:lowerbound:heat}; notice that $R$ is a quantity assumed to be large, by using Taylor approximations we deduce that the leading term of $u_j(t)$ behaves as
\[
u_j(t)\sim \frac{1}{\sqrt{1+2t}}e^{-\frac{R^2}{2(1+2t)}},
\]
which can be bounded uniformly with respect to $t\in(0,1)$. Thanks to the fact that the number of lattice points $j$ such that $R-2<|hj|<R+1$ is comparable to $h^{-1}$ we conclude
$$
h\int_0^1\sum_{\substack{j\in \Z \\ R-2<|hj|<R+1}}(|u_j(0)|^2+ |u_j(t)|^2) \,dt\sim e^{-cR^2}.
$$
If $Rh$ is large, we proceed as in the second part of Theorem \ref{Thm:lowerbound:heat}; similarly we get $
u_j(t)\sim e^{-c\frac{R}{h}\log(Rh)}$,
and therefore
$$
h\int_0^1\sum_{\substack{j\in \Z \\ R-2<|hj|<R+1}}(|u_j(0)|^2+ |u_j(t)|^2) \,dt\sim e^{-c\frac{R}{h}\log(Rh)}.
$$

As in Subsection \ref{subsec:upperb}, the argument can be adapted to the multidimensional case.

\section{Landis-type result for the stationary discrete Schr\"odinger equation}
\label{sec:elliptic}

The strategy carried out to deal with the parabolic setting can be adapted to  discuss the discrete elliptic framework of Landis conjecture. Now we are concerned with the problem
\begin{equation}\label{p2}
\Delta_{\dis} u_j + V_j u_j=0 \quad   \text{in } (h\mathbb{Z})^d,
\end{equation}
where $V:(h\Z)^d\to \R$ is a bounded potential and $u:(h\Z)^d\to \R$.

As before, we will prove a qualitative Landis-type result assuming decay of the solution to \eqref{p2}, which will be based on lower quantitative bounds for the solutions.

\subsection{Lower bound: proof of Theorem \ref{Thm:lowerbound:elliptic}}

The lower bound relies on the Carleman inequality in Theorem \ref{ThCarlemanElliptic-e}. The proof of this inequality follows the same steps as the one for the heat equation in Theorem \ref{ThCarlemanHeat-h} but it is simpler, since there are no conditions between the parameters thanks to the positivity of the commutator. We will omit such a proof.

\begin{thm}[\textbf{Carleman inequality}]\label{ThCarlemanElliptic-e}
Let $R>1$ and $f:(h\mathbb{Z})^d  \to \R$ be such that
$$
\supp(f) \subset \Big\{ 1 \le \big|\frac{hj}{R} + 3 e_1 \big|  \le 4 \Big\}  \times (0,1).
$$
Let $\phi: (h\mathbb{Z})^d \to \mathbb{R}$, $\phi_j = \alpha \big| \frac{hj}{R} +3 e_{1} \big|^2$. Then there exist $R_0=R_0(d)>0$, $C>1$ and $h_0>0$  such that
\begin{multline}\label{Carleman:elliptic}
\frac{1}{h^2} \sqrt{\sinh \frac{ 2\alpha h^2}{R^2}}  \sinh \Big( \frac{2\alpha h}{R\sqrt{d}} \Big) \|f \|_{\ell^2(h\mathbb{Z}^d)}  + 2 \frac{1}{h^2} \sqrt{ \sinh  \frac{2\alpha h^2}{R^2} } \Big( h^d \sum_{j\in \mathbb{Z}^d} \sum_{k=1}^d\Big|\frac{f_{j+e_k} -f_{j-e_k} }{2} \Big|^2  \Big)^{1/2}\\
  \qquad \le C \|e^{\phi} \Delta_{\dis} (e^{-\phi}f) \|_{\ell^2(h\mathbb{Z}^d)}
\end{multline}
for all $R\ge R_0$, $0<h<h_0$.
\end{thm}

The Carleman inequality will lead to the quantitative lower bound stated in Theorem  \ref{Thm:lowerbound:elliptic}.

\begin{proof}[Proof of Theorem \ref{Thm:lowerbound:elliptic}]
The proof is similar to the proof of the lower bound in the parabolic setting and therefore we will omit part of the details. In a first step, we apply the previous Carleman estimate and conclude, analogously as in Lemma \ref{lem:crucial}, that
\begin{equation*}
 \frac{1}{h^4} \sinh \frac{ 2\alpha h^2}{R^2}   \sinh^2 \Big( \frac{2\alpha h}{R\sqrt{d}}\Big)   e^{-  14\alpha}  \le h^d\sum_{\substack{j\in \Z^d\\R-2<|hj|<R+1}} |u_j|^2
\end{equation*}
as long as
\begin{equation}\label{hide:VCelliptic1}
C\left(1 +\|V\|_{\infty}^2\right) \le \frac{1}{h^4} \sinh \frac{ 2\alpha h^2}{R^2}   \sinh^2 \Big( \frac{2\alpha h}{R\sqrt{d}} \Big).
\end{equation}
Analogously as in the parabolic case, let us first link $R$ and $h$, write $R=h^{-\beta}$ for $\beta>0$, and do some casuistic to determine $\alpha$. Since this choice is related to the amount of mass of the nontrivial solution we capture at the scale $|hj|\sim R$, our goal is to find, at each scale, the smallest possible value of $\alpha$, depending on $R$, such that \eqref{hide:VCelliptic1} holds. At the scale $R=h^{-\beta}$ the condition reads
\begin{equation}\label{hide:VCelliptic2}
R^{4/\beta}\sinh\big(2\frac{\alpha}{R^{2+2/\beta}}\big)\sinh^2\big(2\frac{\alpha}{R^{1+1/\beta}\sqrt{d}}\big)   \ge C_0.
\end{equation}
 \noindent{\textbf{Case 1: $\alpha \le R^{1+1/\beta}$.}}  In this first case both hyperbolic functions are small and we can use the chain of inequalities $2x\ge \sinh(x)\ge x$ as we did in the parabolic setting. This implies
\[
R^{4/\beta}\sinh\big(2\frac{\alpha}{R^{2+2/\beta}}\big)\sinh^2\big(2\frac{\alpha}{R^{1+1/\beta}\sqrt{d}}\big)    \sim c\frac{\alpha^3}{R^4}.
\]
Thus, a choice to get a constant size and answer positively to our question is
$$
\alpha=cR^{4/3},
$$
 with an appropriate $c$, independent on $h$ and $R$. This choice is possible only if $4/3 \le 1+1/\beta\Leftrightarrow \beta\le 3$. Notice that $\beta=3$ requires $c\le 1$, which may not be satisfied. If $\beta >3$, this case will not lead to a choice of $\alpha$ such that \eqref{hide:VCelliptic1} holds.\\
\noindent{\textbf{Case 2: $R^{2+2/\beta}\ge\alpha\ge \sqrt{d}R^{1+1/\beta}/2$.}} In this case, by using the asymptotics of the $\sinh$ function
$$
R^{4/\beta}\sinh\big(2\frac{\alpha}{R^{2+2/\beta}}\big)\sinh^2\big(2\frac{\alpha}{R^{1+1/\beta}\sqrt{d}}\big)    \sim \frac{\alpha}{R^{2-2/\beta}}e^{4\frac{\alpha}{R^{1+1/\beta}\sqrt{d}}}.
$$
This motivates the choice
$$
\alpha=cR^{1+\frac{1}{\beta}}\log (R^{1-1/\beta})
$$
with $c$ a positive constant.  Hence,
\[
\frac{\alpha}{R^{2-2/\beta}}e^{\frac{4\alpha}{R^{1+1/\beta}\sqrt{d}}} =c R^{\frac{4c(1-1/\beta)}{\sqrt{d}}+\frac{3}{\beta}-1}\log (R^{1-1/\beta})
\]
and by taking $c=\sqrt{d}/4$ condition \eqref{hide:VCelliptic2} is satisfied. If $\beta\le 3$ we can remove the logarithm from the choice of $\alpha$ and take $\alpha=c R^{2+2/\beta}$ with $c$ a sufficiently large constant, independent of $R,h$ and $\beta$.\\
\noindent{\textbf{Case 3: $\alpha\ge R^{2+2/\beta}$.}} Again, by the asymptotics of $\sinh$,
$$
R^{4/\beta}\sinh\big(2\frac{\alpha}{R^{2+2/\beta}}\big)\sinh^2\big(2\frac{\alpha}{R^{1+1/\beta}\sqrt{d}}\big)    \sim R^{4/\beta}\frac{\alpha}{R^{2-2/\beta}}e^{2\frac{\alpha}{R^{2+2/\beta}}+4\frac{\alpha}{R^{1+1/\beta}\sqrt{d}}}.
$$
And the appropriate choice of $\alpha$ is $R^{2+2/\beta}$.
Gathering the information of these three cases, if $\beta\le 3$ we take $\alpha=\min\{c R^{4/3},c R^{1+1/\beta},R^{2+2/\beta}\}=cR^{4/3}$, while, if $\beta>3$ we take $\alpha=\min\{c R^{1+1/\beta}\log (R^{1-1/\beta}),R^{2+2\beta}\}=c R^{1+1/\beta}\log (R^{1-1/\beta})$, being all constants independent of $R$. After the choice of $\alpha$ is done, we get the lower bound as in the parabolic setting.

If we look at condition \eqref{hide:VCelliptic1} for fixed $h$, then by taking $\alpha=c\frac{R}{h}\log(Rh)$ the condition is satisfied, by following the previous strategy.
\end{proof}

\subsection{Landis-type result for the solution to the discrete elliptic equation}

We conclude this section with the proof of the Landis-type results for the elliptic equation.

\begin{proof}[Proof of Theorem \ref{thm:Schr}]
The proof follows from Theorem \ref{Thm:lowerbound:elliptic} and the assumptions on the solution. Assume $u$ is nontrivial and let
$$
\Lambda_R:=h^d\sum_{\substack{j\in \Z^d\\R-2<|hj|<R+1}} |u_j|^2 .
$$
If the decay assumption in $(1)$ is satisfied, it is easy to check that $\Lambda_R\lesssim e^{-\mu_0 R^{4/3}}$ for all $R$. On the other hand by taking for instance $R=h^{-1/2}$, by the first part of Theorem \ref{Thm:lowerbound:elliptic} we get $\Lambda_R \gtrsim e^{-C R^{4/3}}$. We arrive at a contradiction by letting $R\to \infty$ (and therefore $h\to 0$), choosing $\mu_0 >C$.

If the decay assumption in (2) is satisfied, $\Lambda_R \lesssim e^{-\mu_0 R^{1+1/\beta}\log (R^{1-1/\beta})}$ for all $R$. Next, we get from the second part of  Theorem \ref{Thm:lowerbound:elliptic}  that $\Lambda_R \gtrsim e^{-C R^{1+1/\beta}\log (R^{1-1/\beta})}$ by looking at the scale $R=h^{-\beta}$, and we arrive at a contradiction by letting $R\to \infty$ (and therefore $h\to0$) if $\mu_0$ is large enough.

To finish with the proof, under decay assumption (3), we are dealing with fixed $h>0$ and we get $\Lambda_R\lesssim e^{-\frac{\mu_0}{h}R \log (Rh)}$, while, from Theorem  \ref{Thm:lowerbound:elliptic} , $\Lambda_R \gtrsim e^{-C \frac{R}{h}\log(Rh)}$, arriving to a contradiction by letting $R\to\infty$ if $\mu_0$ is large enough.
\end{proof}

\begin{rmk}
It should be pointed out that the difference between the three parts of Theorem \ref{thm:Schr} is that under the first decay condition we arrive at a contradiction by letting $h\to0$ looking at the values of the nontrivial solution at points of the lattice that shrinks to the continuum, and therefore it has to be understood as a \textit{close-to-continuum} type result; under the second decay condition we are looking at a \textit{spurious} region of the lattice where we connect the \textit{close-to-continuum} region with the \textit{purely discrete} situation, being this the case where we do not require $h\to0$, i. e., the third decay assumption.
\end{rmk}

\subsection{Examples}
\label{sub:examples}

We illustrate the Landis-type result in Theorem \ref{thm:Schr} with some examples in the spirit of  \cite{LM18}. Recall that we are working with the problem
\begin{equation}\label{eq:1d}
  \Delta_{\dis} u(x)+V(x)u(x)=0, \quad x\in (h\Z)^d,
\end{equation}
where  $V:(h\Z)^d \to \R$. Let $|x|_{\infty}:=\max\{|x_1|,\ldots,|x_d|\}$, $x\in (h\Z)^d$.

\begin{lem}
\label{lem:LM}
  Let $h>0$ and $R>0$ be of the form $R=Nh$ with $N\in \mathbb{N}^*$. Let  $V:(h\Z)^d \to \R$ and $u$ be a solution to \eqref{eq:1d}. Denote $q_N:=h^2\max_{|n|_{\infty}=N}|V_n|$, $n\in \N^d$.    If $u(x)$ satisfies the following decay estimate
$$
	\max_{|x|_{\infty}\in \{R-h,R\}}|u(x)|
< \frac{	\max_{|x|_{\infty}\in \{0,h\}}|u(x)|}{(4d-1+q_N) (4d-1+q_{N-1})\cdots (4d-1+q_1)} ,
$$
then $u\equiv 0$.
\end{lem}
\begin{proof}
  The equation can be rewritten as $
\sum_{k=1}^d u(x-he_k)=-\sum_{k=1}^du(x+he_k)+(2d-\widetilde{V}(x))u(x)$,
with $\widetilde{V}(x)=h^2V(x)$.
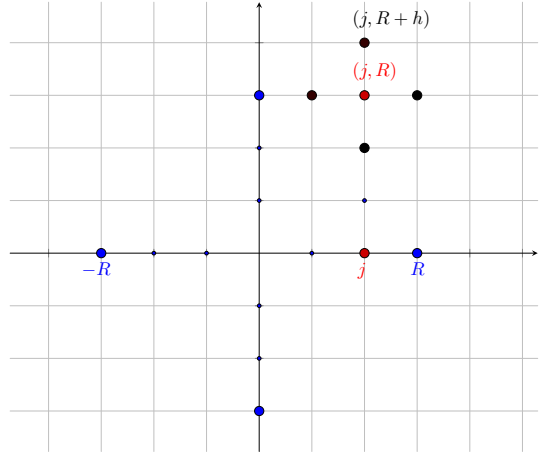
\begin{figure}
\centering
\begin{tikzpicture}[line cap=round,line join=round,>=triangle 45,x=1cm,y=1cm, scale=0.7]
\begin{axis}[
x=1cm,y=1cm,
axis lines=middle,
ymajorgrids=true,
xmajorgrids=true,
xmin=-4.730000000000001,
xmax=5.290000000000004,
ymin=-3.770000000000001,
ymax=4.770000000000001,
xtick={-6,-5,...,11},
ytick={-6,-5,...,6},
 yticklabels={,,},
  xticklabels={,,}]
\clip(-6.73,-6.77) rectangle (11.29,6.77);
\draw [color=qqqqff](-3.49,-0.05) node[anchor=north west] {$-R$};
\draw (1.65,4.75) node[anchor=north west] {$(j,R+h)$};
\draw [color=ffqqqq](1.65,3.75) node[anchor=north west] {$(j,R)$};
\draw [color=qqqqff](2.75,-0.05) node[anchor=north west] {$R$};
\draw [color=ffqqqq](1.75,-0.05) node[anchor=north west] {$j$};
\begin{scriptsize}
\draw [fill=qqqqff] (3,0) circle (2.5pt);
\draw [fill=qqqqff] (-3,0) circle (2.5pt);
\draw [fill=qqqqff] (0,3) circle (2.5pt);
\draw [fill=qqqqff] (0,-3) circle (2.5pt);
\draw [fill=ttqqqq] (2,4) circle (2.5pt);
\draw [fill=ttqqqq] (1,3) circle (2.5pt);
\draw [fill=ccqqqq] (2,3) circle (2.5pt);
\draw [fill=black] (3,3) circle (2.5pt);
\draw [fill=black] (2,2) circle (2.5pt);
\draw [fill=ccqqqq] (2,0) circle (2.5pt);
\draw [fill=qqqqff] (0,2) circle (1pt);
\draw [fill=qqqqff] (0,1) circle (1pt);
\draw [fill=qqqqff] (2,1) circle (1pt);
\draw [fill=qqqqff] (-2,0) circle (1pt);
\draw [fill=qqqqff] (-1,0) circle (1pt);
\draw [fill=qqqqff] (1,0) circle (1pt);
\draw [fill=qqqqff] (0,-1) circle (1pt);
\draw [fill=qqqqff] (0,-2) circle (1pt);
\end{scriptsize}
\end{axis}
\end{tikzpicture}\caption{Illustration in two dimensions; for a point in the ball of center $(0,0)$ and radius $R$ and the metric $|\cdot|_{\infty}$, say $(j,R)$ with $j,R\in h\Z$, $j<R$, and for a solution of the equation, then $u(j,R-h)= 4u(j,R)-u(j-h,R)-u(j+h,R)-u(j,R+h)+Vu(j,R)$ holds. Observe that on the right hand side of this equation there are $2\cdot 2+(2\cdot 2-2)$ evaluations on the ball of radius $R$ ($2\cdot 2$ are coming from the values of $u(j,R)$ and $(2\cdot 2-2)$ from the values of $u(j-h,R), u(j+h,R)$), and one evaluation of $u$ on the ball of radius $R+h$. In higher dimensions, this corresponds to $2d + (2d-2)$ evaluations on the ball of radius $R$ and one on the ball of radius $R+h$.} \label{grid2d}
\end{figure}
Let $y\in (h\mathbb{Z})^d$ such that $|y|_{\infty}=R-h$.
For simplicity we assume the maximum is attained in the first index,
$ |y|_{\infty}=|y_{1}|$.

\noindent\textit{Case $y_1>0$}. Then $y_1=R-h$. Let $x=y+he_{1}=(R,y_2,\dots,y_d)$.  Then $|x|_{\infty}= R$ and $|x + he_1|_{\infty} =R+h$, $|x \pm he_k|_{\infty} =R$ if $k\ne 1$.
Hence
$$
u(y)=u(x-he_1)=-  \sum_{k=2}^d u(x-he_k)- \sum_{ k=2}^d u(x+he_k)+ u (x+he_1)+(2d-\widetilde{V}(x))u(x)
$$
and thus (see Figure \ref{grid2d})
\begin{align*}
	|u(y)| &\le  \sum_{k=2}^d |u(x-he_k)|+ \sum_{ k=2}^d |u(x+he_k)|+  |u(x+he_1)|+|2d-\widetilde{V}(x))|u(x)| \\
	& \le  (2d+2d-2+\widetilde{V}(x)) \max_{|x|=R}|u(x)|  + \max_{|x|_{\infty}=R+h}|u(x)|.
\end{align*}
\noindent \textit{Case $y_1<0$}. Then $y_1=-R+h$. Let $x=y-he_{e_1}=(-R,y_2,\dots,y_d)$.  Then $|x|_{\infty}= R$ and $|x-he_1|=R+h$, $|x \pm he_k|_{\infty} =  R+h$ for $k\ge 2$.  We argue as above writing the expression for $u(x+he_1)$ using the equation
$$
	u(y)=u(x+he_1)=-  \sum_{k=2}^d u(x+he_k)- \sum_{ k=2}^d u(x-he_k)-u(x-he_1)+(2d-\widetilde{V}(x))u(x).
$$
Since $y$ was chosen arbitrary with $|y|_\infty=R-h$, we conclude that (after renaming variables)
$$	
\max_{|x|=R-h}|u(x)| \le  (2d+2d-2+\widetilde{V}(x)) \max_{|x|=R}|u(x)|  + \max_{|x|_{\infty}=R+h}|u(x)|.
$$
Since $u(x)$ with $|x|=R$ also satisfies the previous inequality, then we can conclude that
$$\max_{|x|_{\infty} \in \{ R-h,R\}}|u(x)| \ \le (\max_{|x|_{\infty}=R}|\widetilde{V}(x)|+4d-1)  \max_{|x|_{\infty} \in\{R,R+h\}}|u(x)| .$$

Let $x=nh$, where $n\in \Z^d$. The latter is then rewritten as (here $|n|_{\infty}$ has the same meaning as above, but with the maximum over $\{n_1,\ldots,n_d\}$, $n\in \N^d$)
$$
\max_{|n|_{\infty}\in \{N,N-1\}}|u_n| \le (4d-1+\max_{|n|_{\infty}=N}|\widetilde{V}_n|)\max_{|n|_{\infty}\in \{N,N+1\}}|u_n|.
$$
If we write $M_N:=\max_{|n|_{\infty}\in \{N,N-1\}}|u_n|$ and $q_N:=\max_{|n|_{\infty}=N}|\widetilde{V}_n|$, the recurrence inequality $M_N\le (4d-1+q_N) M_{N+1}$ holds.
Thus, repeating the same procedure
$$
 M_{N+1}\ge\frac{M_{N-1}}{(4d-1+q_N)(2d+1+ q_{N-1})}\ge \cdots \ge\frac{M_1}{(2d+1+q_N)(4d-1+q_{N-1})\cdots (2d+1+q_{1})}.
$$
This gives that  if
$$
\max_{|n|_{\infty}\in \{N+1,N\}}|u_n|=M_{N+1}<\frac{M_1}{(4d-1+q_N)(4d-1+q_{N-1})\cdots (4d-1+q_{1})},$$ then $u\equiv0$. By rescaling in $h$, we obtain the result.
\end{proof}

\begin{cor}
\label{cor:better}
	Let $h>0$ and $R>0$ be of the form $R=Nh$ with $N\in \mathbb{N}^*$. Let $V:(h\Z)^d \to \R$ be a bounded potential. Let $u$ be a solution to \eqref{eq:1d} which satisfies the decay estimate
	\begin{equation*}
		\max_{|x|_{\infty}\in \{R+h,R\}}|u(x)|<  e^{-\frac{R}{h}\log(4d-1+h^2\|V\|_{\infty})}\cdot e^{\max_{|x|_{\infty}\in \{h,0\}}|u(x)|},
	\end{equation*}
	then $u\equiv 0$.
\end{cor}
\begin{proof}
	Since $V$ is bounded, we easily obtain that $q_{n} \le \|\widetilde{V}\|_{\infty}=h^2\|V\|_{\infty}$. Thus  if
	$$	\max_{|x|_{\infty}\in \{R+h,R\}}|u(x)|< \frac{\max_{|x|_{\infty}\in \{h,0\}}|u(x)|}{(\|\widetilde{V}\|_{\infty}+4d-1)^N},$$
	then $u$ satisfies the hypothesis of  Lemma \ref{lem:LM}, and, therefore, $u\equiv 0.$  	Notice that the right hand side of the above  bound can be written as $ e^{-\frac{R}{h}\log(4d-1+\|\widetilde{V}\|_{\infty})} \max_{|x|_{\infty}\in \{h,0\}}|u(x)|$,
	and we are done.
\end{proof}

\begin{rmk}
Observe that Corollary \ref{cor:better} is a scaled version of \cite[Proposition 4.3]{LM18}.
\end{rmk}

As application of Lemma \ref{lem:LM}, we establish a unique continuation result with a specific potential.

\begin{example}\label{Example1} Let $h>0$ and $R>0$ be of the form $R=Nh$ with $N\in \mathbb{N}^*$.  Let $u$ be a solution to
\begin{equation*}\label{eq:V=x}
  \Delta_d u(x)+x u(x)=0, \quad x\in (h\Z)^d.
\end{equation*}
Then, if $u$ satisfies the decay condition
$$
|u(R)|< e^{- \frac{R}{h} \log( 4d-1+h^2R)} \max_{|x|_{\infty}\in \{h,0\}}|u(x)|.
$$
then  $u \equiv 0$.
\end{example}
\begin{proof}
 In this situation we obtain $\widetilde{V}(x)=h^2x $. Recalling the notation from Lemma \ref{lem:LM}, in this case we have $ q_N=h^3 N$ and $q_{N-1}\cdots q_1=(h^3)^{N-1}(N-1)!$.
Thus, by applying Lemma  \ref{lem:LM},  we obtain that, if $u$ satisfies the decay estimate
\begin{equation}
\label{eq:example1}
|u(R)|< e^{-\log( (4d-1+h^3N)(4d-1+h^3(N-1))\dots (4d-1+h^3) )} \max_{|x|_{\infty}\in \{h,0\}}|u(x)|
\end{equation}
then $ u\equiv 0.$
Notice that, a stronger decay condition on $u$ is $$
|u(R)|< e^{-\log( (4d-1+h^3N)^N} \max_{|x|_{\infty}\in \{h,0\}}|u(x)| = e^{- N \log( 4d-1+h^3N)} \max_{|x|_{\infty}\in \{h,0\}}|u(x)|
$$
which can be rewritten as $
|u(R)|< e^{- \frac{R}{h} \log( 4d-1+h^2R)} \max_{|x|_{\infty}\in \{h,0\}}|u(x)|$,
so the result holds.
\end{proof}

\begin{rmk}
Notice that the condition \eqref{eq:example1} implies:
\begin{align*}
|u(R)|< e^{-\log( (h^{3})^NN!)} \max_{|x|_{\infty}\in \{h,0\}}|u(x)|&\le e^{-(3N)\log h-\log(N!)} \max_{|x|_{\infty}\in \{h,0\}}|u(x)|\\
&=e^{-(3N)\log h-N\log N +N} \max_{|x|_{\infty}\in \{h,0\}}|u(x)|.
\end{align*}
For large $N$, assuming $h=1$ for simplicity, the aforementioned weaker decay simplifies to $e^{-N\log N+N}$, which, in accordance with \eqref{eq:jota}, is essentially  the behavior of the Bessel function $|J_N(t)|$ of order $N$ and fixed argument $t$,  as $N \to \infty$. Notably, the potential associated with $J_n(t)$ is a solution to \eqref{eq:1d} is $2\big(\frac{n}{t}-1\big)$ and thus, in this situation,  $|J_n(t)|$ serves as an upper bound for solutions to
$\Delta_d u(n)+ 2\big(\frac{n}{t}-1\big) u(n)=0$.  Hence, it would be interesting to relax the condition  \eqref{eq:example1} which, as we have just seen, it is related to the behaviour of the Bessel function $J_N(t),$ to conclude that $u\equiv 0$.
\end{rmk}

\appendix

\section{Convergence to continuum Gaussian function}
\label{sub:heats}

To support the fact that the weight involved in the decay conditions in Theorem \ref{thm:qualitative} and Theorem~\ref{thm:UpperBound} can be seen as a discrete version of the inverse of the Gaussian function, we rephrase here the convergence result in \cite[Proposition 4.7]{CJK10} as follows.
\begin{prop}
\label{convergence}
Let $h>0$ and $x\in \R$. For any $t>0$ we have
$$
\lim_{h\to 0} \frac{1}{h}e^{-2t/(xh)^2}I_{1/h}(2t/(xh)^2)=\frac{x}{\sqrt{4\pi t}}e^{-\frac{x^2}{4t}}.
$$
\end{prop}
\begin{proof}
The proof is the similar to\cite[Proposition 4.7]{CJK10} with the proper modifications.  We are going to prove that, for $j\in \Z$ and $h>0$ such that $jh=1$, we have
$$
\lim_{\substack{ j\to \infty, h\to 0\\
jh=1}} je^{-2tj^2/x^2}I_{j}(2tj^2/x^2)=\frac{x}{\sqrt{4\pi t}}e^{-\frac{x^2}{4t}}.
$$
We will use the identity $I_n(t)=\frac{1}{\pi}\int_0^{\pi}e^{t\cos\theta}\cos \theta n \,d\theta$, for $n\in \Z$ (see \cite[p. 456]{PBM1}). Then
\begin{align*}
je^{-2t(j/x)^2}I_{j}\big(2t(j/x)^2\big)&=j\frac{e^{-2t(j/x)^2}}{\pi}\int_0^{\pi}e^{2t(j/x)^2\cos\theta}\cos(\theta j)\,d\theta\\
&=\frac{1}{\pi}\int_0^{\pi j}e^{-2t(j/x)^2(1-\cos (y/j))}\cos y\,dy,
\end{align*}
where we made the change of variable $y=\theta j$ in the last equality.
By Taylor, for all $v\in [0,\pi]$
$$
0<c:=\frac12-\frac{\pi^2}{24}\le \frac12-\frac{v^2}{24}\le \frac{1-\cos v}{v^2}.
$$
From here we obtain the uniform bound $
\frac{1-\cos( y/j)}{(x/j)^2}\ge c\big(\frac{y}{x}\big)^2$.
Moreover $
\frac{1-\cos( y/j)}{(x/j)^2}\xrightarrow{j\to \infty} \frac{y^2}{2x^2}$.
Thus
$$
\int_0^{\pi j}|e^{-2t(j/x)^2(1-\cos (y/j))}\cos y|\,dy\le \int_0^{\infty}e^{-2t \frac{c y^2}{x^2}}\,dy=\frac{\sqrt{\pi} x}{2\sqrt{2c}t^{1/2}},
$$
so that by dominated convergence we have
$$
je^{-2t(j/x)^2}I_{j}\big(2t(j/x)^2\big)\xrightarrow{j\to \infty} \frac{1}{\pi}\int_0^{\infty}e^{-\frac{ty^2}{x^2}}\cos y\,dy=\frac{x}{\sqrt{4\pi t}}e^{-\frac{x^2}{4t}},
$$
where the last identity can be found e.g. in \cite[p. 451]{PBM1}.
\end{proof}

\begin{rmk}
The argument in the proof of Proposition \ref{convergence} can be extended to the multidimensional case without effort. Indeed, we consider the multidimensional discrete weight
$$
\frac{1}{h}e^{-\frac{2t}{h^2}\big(\frac{1}{x_1^2}+\cdots\frac{1}{x_d^2}\big)}\prod_{i=1}^dI_{1/h}(2t/(x_ih)^2), \quad \mathbf{x}=(x_1,\ldots,x_d)\in \R^d.
$$
It holds, for $\mathbf{x}=(x_1,\ldots,x_d)\in \R^d$,
$$
je^{-2tj^2\big(\frac{1}{x_1^2}+\cdots\frac{1}{x_d^2}\big)}\prod_{i=1}^dI_{j}(2t/(x_ih)^2)=j\frac{1}{\pi}\prod_{i=1}^d\int_0^{\pi}e^{-2tj^2\frac{1}{x_i^2}}e^{2t\cos\theta_i}\cos(\theta_i j)\,d\theta_i,
$$
and each integral can be treated separately as in the one dimensional case. Observe also that we could also consider a multidimensional mesh with different scale in every dimension, say $\mathbf{h}=(h_1,\ldots,h_d)$, and then study the limits $h_i\to0$ simultaneously.
\end{rmk}

\section{Some facts on Bessel functions}
\label{sub:Bessel}

Let
$I_{\nu}(z)$ be the modified Bessel function of first kind given by the formula (see \cite[Chapter 5, Section 5.7]{Lebedev})
\begin{equation}
\label{eq:Inu}
I_{\nu}(z)=\sum_{k=0}^{\infty}\frac{(z/2)^{\nu+2k}}{\Gamma(k+1)\Gamma(k+\nu+1)}, \qquad |z|<\infty, \quad |\operatorname{arg} z|<\pi
\end{equation}
and let $ K_{\nu} $ be the Macdonald's function of order $\nu$ (see also \cite[Chapter 5, Section 5.7]{Lebedev})
\begin{equation}
\label{eq:Knu}
K_{\nu}(z)=\frac{\pi}{2}\frac{I_{-\nu}(z)-I_{\nu}(z)}{\sin \nu\pi}, \qquad |\operatorname{arg} z|<\pi, \quad  \nu\neq 0,\pm1, \pm 2, \ldots.
\end{equation}
For integer $\nu=n$, we have that $K_n(z)=\lim_{\nu\to n}K_{\nu}(z)$, $n= 0,\pm1, \pm 2, \ldots$. From \eqref{eq:Knu},
\begin{equation}
\label{eq:par}
K_{-\nu}(z)=K_{\nu}(z).
\end{equation}

Macdonald's functions $K_{\nu}(z)$ satisfy
\begin{equation}
\label{eq:recurrence}
\frac{\partial}{\partial z}K_{\nu}(z)=-\frac{1}{2}(K_{\nu+1}(z) + K_{\nu-1}(z)).
\end{equation}
They are increasing with respect to order (see \cite[(10.37)]{OlMax}), namely for fixed argument,
\begin{equation}\label{monotonicityBessel}
K_\nu(z) < K_\mu(z),   \quad \text{ for all } 0\le \nu <\mu, \quad \text{ for all }  z\in \mathbb{R}.
\end{equation}
They also fulfill the so-called Tur\'an-type inequality (see for instance \cite[Theorem 8]{S}),
\begin{equation}
\label{eq:Turan}
K_{\nu} ^2(z) \le K_{\nu+1}(z) K_{\nu-1}(z), \quad \text{ for all } \nu\in\R, \,\,\, z>0
\end{equation}
and (see \cite[Theorem 2]{Baricz})
\begin{equation}\label{Turan2}
\frac{(K_{\nu}(z))^2}{K_{\nu+1}(z) K_{\nu-1}(z)} \ge \frac{z}{1+z}, \quad |\nu|\ge \frac{1}{2}, \,\,\, z>0.
\end{equation}
From \cite[Theorem 2]{Baricz} by taking $\nu=0$ we deduce
\begin{equation}\label{Turan3a}
\frac{1}{1+\frac{1}{z}+\frac{1}{4z^3}} < \frac{K_0^2(z)}{K_{-1}(z)K_1(z)}<\frac{z}{z+1}, \,\,\, z>0.
\end{equation}
The following monotonicity property of the product of Macdonald's and
modified Bessel functions is proven in \cite[Theorem 1]{S}
\begin{equation}
\label{eq:monotonicity}
\frac{I_{\nu+1/2}(x)}{I_{\nu-1/2}(x)}<\frac{x}{\nu+\sqrt{\nu^2+x^2}}\le \frac{K_{\nu-1/2}(x)}{K_{\nu+1/2}(x)}.
\end{equation}

We have the following asymptotics:
\begin{itemize}
\item for large order and fixed $z\neq 0$, see \cite[(10.41.2)]{OlMax}
\begin{equation}
\label{eq:asympLargeOrd}
K_{\nu}(z) \sim \sqrt{\frac{\pi}{2\nu}} \Big(\frac{e z}{2\nu}  \Big)^{-\nu}\quad   \nu\to \infty.
\end{equation}
\item for all $\nu \in \mathbb{C}$ fixed and large argument, see \cite[(10.40.2)]{OlMax}
\begin{equation}
\label{eq:K0large}
K_\nu(z) \sim \frac{1}{\sqrt{z}}e^{-z}  \quad \text{ as } z\to \infty,
\end{equation}
\item uniform estimate, see \cite[ (10.41.4)]{OlMax}
\begin{equation}
\label{eq:Kjlarge}
K_{\nu}(\nu z) \sim \frac{1}{\sqrt{\nu}(1+z^2)^{1/4}}e^{-\nu\big(\sqrt{1+z^2}+\log\frac{z}{1+\sqrt{1+z^2}}\big)} \quad \text{ as } \nu \to \infty, \text{ uniformly in } z.
\end{equation}
\end{itemize}
On the other hand, for the function $I_{\nu}(x)$ we also have the uniform estimate (\cite[(10.41.3)]{OlMax})
\begin{equation}
\label{eq:Ijlarge}
I_{\nu}(\nu z)\sim \frac{1}{\sqrt{2\pi}\nu^{1/2}(1+z^2)^{1/4}}
e^{\nu\big(\sqrt{1+z^2}+\log\frac{z}{1+\sqrt{1+z^2}}\big)}\quad \text{ as } \nu \to \infty, \text{ uniformly in } z.
\end{equation}
The integral representation for Macdonald's functions is valid, see for instance \cite[10.32.9]{OlMax}
\begin{equation}\label{eq:RepK}
K_{\nu}(z)=\int_0^\infty e^{-{z\cosh(t)}}\cosh(\nu t)\,dt, \quad |\operatorname{ph}z|<\frac12\pi, \quad \nu \text{ complex parameter}.
\end{equation}

Finally, we recall the Bessel function of the first kind of order $\nu$ (\cite[Chapter 5, Section 5.3]{Lebedev})
$$
J_{\nu}(z)=\sum_{k=0}^{\infty}(-1)^k\frac{(z/2)^{\nu+2k}}{\Gamma(k+1)\Gamma(k+\nu+1)}, \qquad |z|<\infty, \quad |\operatorname{arg} z|<\pi,
$$
 with the asymptotics (see \cite[10.19.1]{OlMax})
\begin{equation}
\label{eq:jota}
J_{\nu}(z) \sim \frac{1}{2\pi\nu}\Big(\frac{e z}{2\nu}  \Big)^{\nu}\quad   \nu\to \infty.
\end{equation}

\end{document}